\documentclass[reqno]{amsart}
\usepackage{amsmath}
\usepackage{amssymb}
\usepackage{graphicx}
\usepackage{enumerate}

\newtheorem{theorem}{Theorem}[section]
\newtheorem{lemma}{Lemma}[section]
\newtheorem{prop}{Proposition}[section]
\newtheorem{alg}{Algorithm}[section]
   
\newcommand{\abs}[1]{\left\lvert#1\right\rvert}
\newcommand{\ceil}[1]{\lceil#1\rceil}
\newcommand{\floor}[1]{\lfloor#1\rfloor}
\newcommand\PTE{\textsc{PTE}}
\newcommand\UFD{\textsc{UFD}}
\newcommand\TS{\rule{0pt}{2.6ex}}
\newcommand\BS{\rule[-1.2ex]{0pt}{0pt}}
\newcommand{\gfextn}[0]{pdf}
\DeclareMathOperator{\ord}{ord}

\newenvironment{iolist}[1]%
{\begin{list}{}{%
\settowidth{\labelwidth}{\textsf{{\it #1.}}}%
\setlength{\labelsep}{4mm}%
\setlength{\leftmargin}{\labelwidth}%
\addtolength{\leftmargin}{\labelsep}%
}}%
{\end{list}}

\newenvironment{biglabellist}[1]%
{\begin{list}{}{%
\settowidth{\labelwidth}{\textsf{{\it #1.}}}%
\setlength{\labelsep}{2mm}%
\setlength{\leftmargin}{\labelwidth}%
\addtolength{\leftmargin}{\labelsep}%
\addtolength{\leftmargin}{4mm}%
\setlength{\itemsep}{6pt}%
\setlength{\listparindent}{0pt}%
\setlength{\topsep}{3pt}%
}}%
{\end{list}}

\begin{document}

\title{Ideal Solutions in the Prouhet--Tarry--Escott problem}

\dedicatory{Dedicated to the memory of Peter Borwein}

\author[D. Coppersmith]{Don Coppersmith}
\address{Center for Communications Research, Princeton, NJ USA}
\author[M.~J. Mossinghoff]{Michael J. Mossinghoff}
\address{Center for Communications Research, Princeton, NJ USA}
\email{m.mossinghoff@idaccr.org}
\author[D. Scheinerman]{Danny Scheinerman}
\address{Center for Communications Research, Princeton, NJ USA}
\email{daniel.scheinerman@gmail.com}
\author[J.~M. VanderKam]{Jeffrey M. VanderKam}
\address{Center for Communications Research, Princeton, NJ USA}
\email{vanderkm@idaccr.org}

\date{\today}

\begin{abstract} 
For given positive integers $m$ and $n$ with $m<n$, the
\textit{Prouhet--Tarry--Escott problem} asks if there exist two disjoint
multisets of integers of size $n$ having identical $k$th moments for $1\leq
k\leq m$; in the \textit{ideal} case one requires $m=n-1$, which is
maximal.
We describe some searches for ideal solutions to the
Prouhet--Tarry--Escott problem, especially solutions possessing a
particular symmetry, both over $\mathbb{Z}$ and over the ring of integers
of several imaginary quadratic number fields.
Over $\mathbb{Z}$, we significantly extend searches for symmetric ideal
solutions at sizes $9$, $10$, $11$, and $12$, and we conduct extensive
searches for the first time at larger sizes up to $16$.
For the quadratic number field case, we find new ideal solutions of sizes
$10$ and $12$ in the Gaussian integers, of size $9$ in
$\mathbb{Z}[i\sqrt{2}]$, and of sizes $9$ and $12$ in the Eisenstein
integers.
\end{abstract}
                                                                                
\maketitle

\section{Introduction}\label{secIntro}

The \textit{Prouhet--Tarry--Escott} problem is a classical problem in
Diophantine analysis.
In its original form, it asks for disjoint multisets of integers
$A=\{a_1,\ldots,a_n\}$ and $B=\{b_1,\ldots,b_n\}$ having identical $k$th
moments for all positive integers $k$ up to a given value $m$:
\begin{equation}\label{eqnPTE1}
\sum_{i=1}^n a_i^k = \sum_{i=1}^n b_i^k,\quad 1\leq k\leq m.
\end{equation}
The integer $n$ is known as the \textit{size} of the problem, and $m$ its
\textit{degree}.
This is easily seen to be equivalent to either of the following two
conditions involving polynomials one may construct from $A$ and $B$:
\begin{gather}
\label{eqnPTE2}
\deg\left(\prod_{i=1}^n (x-a_i) - \prod_{i=1}^n (x-b_i)\right) < n-m,\\
(x-1)^{m+1} \;\Big\vert\; \sum_{i=1}^n \left(x^{a_i} - x^{b_i}\right). \nonumber
\end{gather}
From \eqref{eqnPTE2} we see that the largest possible value for the degree
$m$ is $n-1$: the case where equality is achieved here is known as the
\textit{ideal} case of the Prouhet--Tarry--Escott problem.
We focus on the ideal case in this article.

This problem in various forms dates back more than $250$ years: a special
case arose in correspondence between Euler and Goldbach around 1750, where
it was noted that $A=\{a,b,c,a+b+c\}$ and $B=\{0,a+b,a+c,b+c\}$ forms a
solution of degree $2$.
Tarry and Escott studied the more general problem in the 1910s, obtaining a
number of results on it, and consequently the problem is often referred to
as the \textit{Tarry--Escott} problem in the literature.
Since Wright \cite{Wright59} pointed out a contribution of Prouhet from 1851
on this topic, the latter name is usually attached as well in current
treatments of this topic.
We follow this and refer to the problem as the Prouhet--Tarry--Escott
problem, or \PTE\ problem for short.
Other names occur in the literature as well: Hua \cite{Hua} labels it ``the
problem of Prouhet and Tarry,'' Dickson \cite{Dickson} as ``equal sums of
like powers,'' and a number of authors (e.g., \cite{Smyth91}) refer to
solutions of \eqref{eqnPTE1} as \textit{multigrades}.
For additional background on this problem, we refer the reader to
\cite[ch.~11]{BorweinCEANT}, \cite{BI}, \cite{Chernick},
\cite[ch.~24]{Dickson}, \cite{DB37}, \cite{Gloden}, or
\cite[\S18.7]{Hua}.

It is easy to verify that if $A$ and $B$ form a solution to the \PTE\
problem with size $n$ and degree $m$, then any nontrivial affine 
transformation of these sets $rA+s$, $rB+s$ also forms such a solution, so
one may normalize solutions by requiring for instance that $0$ is an
element of one of the sets, as in the Euler--Goldbach example.

Much study of the \PTE\ problem concentrates on the special class of
\textit{symmetric} solutions.
When $n$ is even, a symmetric solution has the property that $A=-A$ and
$B=-B$, so that all of the odd moments of $A$ and $B$ are $0$ and thus
automatically match.
When $n$ is odd, a symmetric solution has $B=-A$, so that all of the even
moments of these two sets automatically match.

Ideal solutions in the \PTE\ problem over $\mathbb{Z}$ are known for $n\leq
10$ and $n=12$.
Infinite families of such solutions are known in each of these cases except
$n=9$, where just two solutions are known up to affine equivalence, both
found by Letac in 1942 \cite{Letac42} (see also
\cite[pages~47--48]{Gloden}):
\begin{equation}\label{eqnPTE9}
\begin{split}
A &= \{-98, -82, -58, -34, 13, 16, 69, 75, 99\},\; B=-A;\\
A &= \{-174, -148, -132, -50, -8, 63, 119, 161, 169\},\; B=-A.
\end{split}
\end{equation}
See \cite{BorweinCEANT}, as well as \cite{Choudhry17} and its references,
for information on parameterized families of ideal solutions (and symmetric
ideal solutions) with size $n\leq8$.

For $n=10$, Smyth \cite{Smyth91} combined a construction of Letac with the
theory of elliptic curves to determine an infinite family of non-equivalent
symmetric ideal solutions of this size.
Most of the solutions produced by this method contain members that are
quite large, and since then some smaller solutions have been found using
computational strategies.
In 2002, Borwein, Lisonek, and Percival \cite{BLP} developed an algorithm
that found two symmetric ideal solutions at $n=10$, both smaller than any
prior known solution by two orders of magnitude:
\begin{equation}\label{eqnPTE10}
\begin{split}
A &= \pm\{99,100,188,301,313\},\; B=\pm\{71,131,180,307,308\};\\
A &= \pm\{103,189,366,452,515\},\; B=\pm\{18,245,331,471,508\}.
\end{split}
\end{equation}
They also proved that no symmetric ideal solutions with size $n=11$ exist
with height at most $2000$ (by the \textit{height} of a solution we mean
its largest element in absolute value), and no additional symmetric ideal
solutions with size $n=9$ exist up to this same height, beyond those
equivalent to one of the solutions~\eqref{eqnPTE9}.

The first ideal solution with size $n=12$ was discovered in 1999 by Kuosa,
Meyrignac, and Shuwen, announced on a website maintained by Meyrignac.
They found the symmetric solution
\begin{equation}\label{eqnPTE12a}
A=\pm\{22,61,86,127,140,151\},\; B=\pm\{35,47,94,121,146,148\}.
\end{equation}
Borwein et al.\ \cite{BLP} verified that no other symmetric solutions of
this size exist with height at most $1000$.
A second symmetric solution was discovered by Broadhurst in 2007
\cite{Broadhurst}:
\begin{equation}\label{eqnPTE12c}
\begin{split}
A&=\pm\{257,891,1109,1618,1896,2058\},\\
B&=\pm\{472,639,1294,1514,1947,2037\}.
\end{split}
\end{equation}
In 2008, Choudhry and Wr\'ob{\l}ewski \cite{CW} employed an elliptic
curve to construct an infinite family of symmetric ideal \PTE\ solutions at
$n=12$, and used their parameterization to search for small solutions at
this size.
They found six more solutions with height less than $10^{10}$, the smallest
of which is
\begin{equation}\label{eqnPTE12b}
\begin{split}
A&=\pm\{107,622,700,1075,1138,1511\},\\
B&=\pm\{293,413,886,953,1180,1510\};
\end{split}
\end{equation}
the next-smallest has height $14770$.

The \PTE\ problem has also been investigated in other domains.
In 2007, Alpers and Tijdeman \cite{AT07} introduced a multi-dimensional
version of this problem, concentrating on the case $\mathbb{Z}^2$.
They also suggested investigating this problem over the Gaussian integers
$\mathbb{Z}[i]$, noting that solutions occur there that are not equivalent
to solutions over $\mathbb{Z}$, nor inherited in a natural way from their
work over $\mathbb{Z}^2$.
A few years later, Caley  \cite{Caley12,Caley13} developed the \PTE\
problem over the ring of integers of a number field, especially when this
ring is a unique factorization domain (\UFD).
This work included adapting the method of \cite{BLP} to search for ideal
solutions with size $n\leq12$ in the Gaussian integers, and Caley presented
new solutions over this ring for $6\leq n\leq 10$.
His solutions for $n=10$ appear in Section~\ref{secCompQ}.

In this article, we investigate the \PTE\ problem, both in the classical
setting over $\mathbb{Z}$, and over certain imaginary quadratic number
fields where the ring of integers is a \UFD\@.
We develop a new method for searching over $\mathbb{Z}$, and use this to
search for symmetric ideal solutions with sizes between $n=9$ and $n=16$.
For sizes $n=9$ through $n=12$, our searches cover a significantly larger
space compared to the work of \cite{BLP}.
For example, at $n=9$ we complete a search for solutions with height at most
$7000$, compared to $2000$ in \cite{BLP}.
(See Table~\ref{tableZComp} in Section~\ref{subsecOddSearches} for a
comparison of height bounds in the other cases.)
For larger $n$, it appears that no significant searches had been performed
in prior work.
No new integral solutions are found for $9\leq n\leq 16$, beyond ones
equivalent to known configurations.
We also remark on some interesting solutions found that consist mostly of
integers, but also contain a few algebraic irrationals of small degree.

Our detection of certain ``near misses'' in the \PTE\ problem, where
solutions contain mostly integers, along with a few algebraic integers from
a quadratic extension of $\mathbb{Q}$, provides some motivation toward
investigating the \PTE\ problem in some simple number fields.
If our searches over $\mathbb{Z}$ uncovered solutions in the ring of
integers of number fields such as $\mathbb{Q}(\sqrt{14})$ and
$\mathbb{Q}(\sqrt{231})$, perhaps we can find solutions in simple rings
like the Gaussian or Eisenstein integers using methods specifically
engineered for such settings.

In the setting of imaginary quadratic number fields, we concentrate on the
cases $\mathbb{Q}(i\sqrt{d})$ with $d\in\{1,2,3,7,11,19\}$, so this
includes the Gaussian integers and the Eisenstein integers (where $d=3$).
Here, we identify some additional local structure for symmetric ideal \PTE\
solutions, and use this to conduct more extensive searches over the
Gaussian integers compared to \cite{Caley13}, as well as new searches over
the other five rings.
We also identify some families of ideal solutions in local settings, which
present obstructions to establishing certain arithmetic information about
solutions over these number fields.

This article is organized in the following way.
Section~\ref{secDivis} defines the \textit{constant} in an ideal \PTE\
solution, and describes some arithmetic requirements regarding its value in
ideal and symmetric ideal solutions.
Section~\ref{secLocal} establishes some families of ideal solutions in the
\PTE\ problem in local settings, which obstruct certain additional
arithmetic properties in this constant.
Aided by knowledge of this constant and of local \PTE\ solutions,
Section~\ref{secCompZ} then describes our algorithm for searching for
symmetric ideal solutions over the integers in the \PTE\ problem, and the
results of our searches with $9\leq n\leq 16$, including some solutions
consisting of integers plus a small number of algebraic irrationals of low
degree.
After this, Section~\ref{secDivisQ} determines local information on the
\PTE\ constant over certain imaginary quadratic number fields.
Finally, Section~\ref{secCompQ} describes our algorithms for searching for
symmetric ideal solutions in this setting, and summarizes the results of
our searches there with $9\leq n\leq 16$.
Our results include new solutions over the Gaussian integers at $n=10$ and
$n=12$, over $\mathbb{Z}[i\sqrt{2}]$ at $n=9$, and over the Eisenstein
integers at $n=9$ and $n=12$.

\section[Divisibility Requirements in $\mathbb{Z}$]{Divisibility Requirements in \protect{\boldmath{$\mathbb{Z}$}}}\label{secDivis}

If $A$ and $B$ comprise an ideal solution of size $n$ in the
\PTE\ problem, then the polynomial in~\eqref{eqnPTE2} has degree $0$, and
we denote the value of this constant by $C_n(A,B)$:
\begin{equation}\label{eqnCnZ}
C_n(A,B) = \prod_{i=1}^n (x-a_i) - \prod_{i=1}^n (x-b_i).
\end{equation}
This integer is highly composite, since for each $j=1,\ldots,n$ one has
\begin{equation}\label{eqnCn}
C_n(A,B) = \prod_{i=1}^n (b_j-a_i) = -\prod_{i=1}^n (a_j-b_i).
\end{equation}
In fact, it is known that $C_n(A,B)$ must possess certain arithmetic
properties.
To describe these, it is convenient to define
\[
C_n = C_n(\mathbb{Z}) = \gcd\{C_n(A,B) : \textrm{$(A,B)$ an
ideal \PTE\ solution over $\mathbb{Z}$ of size $n$}\}.
\]
and
\begin{equation*}
\begin{split}
C_n' = C_n'(\mathbb{Z}) = \gcd\{C_n(A,B) :\; &\textrm{$(A,B)$ a
symmetric ideal \PTE}\\
&\quad\textrm{solution over $\mathbb{Z}$ of size $n$}\}.
\end{split}
\end{equation*}
We refer to $C_n$ (and $C_n'$) as \textit{the constant} associated with an
ideal (respectively symmetric ideal) \PTE\ solution of size $n$ over
$\mathbb{Z}$.
We describe some specific arithmetic requirements for these values below.

\subsection{General Requirements}\label{subsecGenlReqs}

In 1975, Kleiman~\cite{Kleiman} proved that $(n-1)!\mid C_n$, and in 1990
Rees and Smyth~\cite{ReesSmyth} established a number of additional facts
regarding required divisors of this constant.
We summarize these in the following proposition.

\begin{prop}[Rees and Smyth \cite{ReesSmyth}]\label{propRS}
Let $C_n$ denote the constant for an ideal \PTE\ solution of size $n$, and
let $p$ denote a prime number.
\begin{enumerate}[(i)]
\item If $k$ is a positive integer and $p<n/k$ then $p^{k+1}\mid C_n$.
\item If $p\geq5$ then $p\mid C_p$.
\item If $n+2\leq p < n+2+\frac{n-3}{6}$ then $p\mid C_n$.
\item $16\mid C_5$, $32\mid C_6$, $11\mid C_7$, $11\cdot13\mid C_8$,
$13\mid C_9$, and $17\mid C_{11}$.
\end{enumerate}
\end{prop}

The third statement in this lemma, as well as most of the fourth, are
deduced from what the authors call the \textit{Multiplicity Lemma}.
This records a requirement for a local analogue of an ideal solution in the
\PTE\ problem.
We record its statement here for the convenience of the reader.
We first require some definitions.
For a polynomial $q\in\mathbb{F}_p[x]$ and $a\in\mathbb{F}_p$, let
$\nu_a(q)$ represent the multiplicity of the root at $x=a$ of $q(x)$.
Also, for an integer $a$ and prime $p$, let $\langle a\rangle_p$ denote
the integer congruent to $a$ mod $p$ in $(-p/2,p/2]$.

\begin{lemma}[Multiplicity Lemma \cite{ReesSmyth}]\label{lemRSML}
If $p$ is prime and $p>n$, $q_1$ and $q_2$ are monic polynomials in
$\mathbb{F}_p[x]$ that split completely over $\mathbb{F}_p$, and $q_1-q_2$
is a nonzero constant $C\not\equiv0\bmod p$, then there exists a
polynomial $h\in\mathbb{F}_p[x]$ having degree exactly $p-n-1$ that
satisfies
\[
h(a) \equiv \nu_a(q_1) - \nu_a(q_2) \pmod p
\]
for each $a\in\mathbb{F}_p$.
Furthermore,
\[
\sum_{a=0}^{p-1} \abs{\langle h(a)\rangle_p} \leq 2n.
\]
\end{lemma}

This result may be used to establish required divisors in the \PTE\ constant
$C_n$: if there exists an ideal \PTE\ solution $(A,B)$ of size $n$ with\linebreak 
$p\nmid C_n(A,B)$, then there exists $h\in\mathbb{F}_p[x]$ as in the lemma.
Thus, if $p>n$ and $\sum_{a=0}^{p-1} \abs{\langle h(a)\rangle_p} > 2n$ for
every polynomial $h\in\mathbb{F}_p[x]$ with degree $p-n-1$, then
necessarily $p\mid C_n$.
Rees and Smyth employed this method to establish the statements regarding
odd prime divisors in $C_n$ for $n\in\{7,8,9,11\}$ in
Proposition~\ref{propRS}, and in his Ph.D. thesis Caley \cite{Caley12} used
it to show that $17 \mid C_{10}$, $19 \mid C_{11}$, and $17\cdot19 \mid
C_{12}$.
Caley also proved there that $2^6\mid C_7$ with a separate elementary
argument.
More recently, Filaseta and Markovich \cite{FM17} used Newton polygons to
obtain $2$-adic information on these constants, and showed that
$2^6 \mid C_8$ and $2^9 \,\|\, C_9$.

We employed the computational strategy of Rees and Smyth in a C++ program
to prove that $23 \mid C_{10}$; prior to this, it was known only that
$23\mid C_{10}'$.
For this computation, using symmetries noted by Caley \cite{Caley12}, it
sufficed to test all polynomials $h(x)$ with degree $12$ in
$\mathbb{F}_{23}[x]$ having linear coefficient $0$ or $1$.
We found that the minimal value of $\sum_{a=0}^{22} \abs{\langle
h(a)\rangle_p}$ over this set is $24$, achieved by $92$ polynomials,
including $x^{12} + x^{10} + x^8 - 7x^6 + 5x^4 - 4x^2$.

We also used this method to determine a number of required divisors in
ideal solutions in the \PTE\ problem for $13\leq n\leq 20$.
A summary of the known required divisors in the constant for ideal
solutions for a number of sizes $n$ is displayed in the second column of
Table~\ref{tableReqDivZ}.
New information determined here, including the determination that $23\mid
C_{10}$, as well as a number of required factors for sizes $n\geq13$, is shown
in boxed form in this table.

\subsection{Additional Requirements for Symmetric Solutions}\label{subsecSymm}

We can deduce additional arithmetic requirements for $C_n'$.
Let $p$ be a prime number, and suppose $n$ is odd.  If every symmetric
ideal solution $(A,-A)$ of size $n$ over $\mathbb{F}_p$ has $A\equiv
-A\bmod p$, then, in view of \eqref{eqnCn}, we may conclude that
$p\mid C_n'$.  Similarly, if $n$ is even, and every symmetric ideal
solution $(A,B)$ of size $n$ over $\mathbb{F}_p$ (so where $A=-A$ and
$B=-B$) has $A\equiv B\bmod p$, then we conclude $p\mid C_n'$.
Borwein et al.~\cite{BLP} used this strategy to deduce the required
divisors shown in the third column of Table~\ref{tableReqDivZ} for
$n=11$, $12$, and $13$.

We determined the additional divisors for $C_n'$ for the cases $14\leq
n\leq20$ with some new searches.
We describe our strategy here when $n$ is odd; the even case is similar.
We wish to determine if there exists a multiset $A$ drawn from
$\mathbb{F}_p$ with the property that $A\neq -A$, and whose odd moments $m$
from $1$ to $n-2$ are all $0$ mod $p$.
A few observations allow us to trim the search space substantially.
Suppose $A=\{a_1,\ldots,a_n\}$ has the desired properties, and $b\in A\cap
-A$.
If $b=0$ then $x$ is an algebraic factor of the polynomial
\[
\prod_{i=1}^n (x-a_i) - \prod_{i=1}^n(x+a_i),
\]
and if $b\neq0$ then $x^2-b^2$ is a factor.
In either case, since this polynomial has degree $0$ in view of the
assumption that $A$ is a symmetric ideal solution in \PTE\ mod $p$, the
remaining cofactor must be identically $0$ mod $p$.
Let $A'=A\setminus\{0\}$ in the first case and $A\setminus\{b,-b\}$ in the
second (removing just a single copy of each element in the case they occur
several times).
It follows that if $b'\in A'$ then $-b'\in A'$, so $A=-A$.
We can therefore assume that $A$ and $-A$ are disjoint in our search, and
in particular that $0\not\in A$.

We can trim the search space further by using the fact that a \PTE\
solution remains a solution after being multiplied by a nonzero constant.
It follows readily that we may assume $1\in A$.   
Some additional work shows that we can often also assume that $\pm2\notin
A$: suppose that $A$ is a symmetric solution of size $n$ with the property
that $2x\in A$ whenever $x\in A$, with matching multiplicities.
Then $\ord_p(2)\mid n$, and by normalizing we can assume that $A$ contains
the consecutive powers of $2$ mod $p$ beginning with $1$, all with the same
multiplicity.
(This multiplicity must in fact be $1$, since otherwise we would have a set
of at most $n/3$ elements with too many vanishing odd moments in view of
\eqref{eqnPTE2}.)
Thus, if $\ord_p(2)\nmid n$, then we may assume that $\pm2\not\in A$.
If $\ord_p(2) = n$, then we need to test the set $\{1,2,4,\ldots,2^{n-1}\}$
separately before searching with $\pm2\not\in A$.
Solutions of this form can arise: at $n=11$ with $p=23$, we have
$\ord_{23}(2)=11$, and the set $A=\{1,2,3,4,6,8,9,12,13,16,18\}$ consisting
of the powers of $2$ mod $23$ forms a symmetric ideal solution in this
local setting.

If $\ord_p(2)$ is a proper divisor of $n$, then one needs to include $2$ in
the search.
(One needs at least a separate search where $A$ is seeded with the powers
of $2$ mod $p$.)
Solutions in this case with $2A\subseteq A$ can occur.
For example, at $n=15$ the set $A=\{1,2,4,5,7,8,9,10,14,16,18,19,20,25,28\}$
forms a symmetric ideal \PTE\ solution mod $31$ with this structure: here
$\ord_{31}(2)=5$.
However, in most cases we can eliminate $\pm2$ after a short check.

In a similar way, when $n$ is even, we can assume that $A$ and $B$ are
disjoint, $A=-A$, $B=-B$, $B$ occurs lexicographically after $A$, and $1\in
A$ (unless $A$ is entirely $0$, in which case $1\in B$).

We employed two strategies in these searches: an exhaustive search for
local solutions subject to the constraints listed here, and a
meet-in-the-middle approach, which was significantly faster.
For the latter, suppose $n$ is odd.
We first construct all multisets $C$ of size $\floor{n/2}$ and store
their odd moments up to $n-2$.
We then consider all multisets $D$ of size $\ceil{n/2}$, checking if any
have odd moments that are the negations mod $p$ of those of a stored
multiset.
If this occurs, then the odd moments of interest of $A=C\cup-D$ are all
$0$, and we can check if $A\neq-A$ mod $p$.
A similar method applies when $n$ is even, where we check that the even
moments of $C$ and $D$ match, and test if $A=C\cup -C$ and $B=D\cup
-D$ differ mod $p$.

For some larger primes $p$ we also considered a variant of the last method
which allowed us to exclude some additional primes as potential required
divisors.
If a solution $A=C\cup-D$ of odd size $n$ over $\mathbb{F}_p$
exists, we can certainly require that the largest (least nonnegative)
residue in $C$ is less than or equal to the smallest one in $D$.
One might expect that values across $\{1,\ldots,p-1\}$ would appear in a
typical solution $A$, so we can use our meet-in-the-middle approach with
$C$ and $D$ drawn respectively from sets of small and large residues mod
$p$.
By selecting bounding residue values near $p/2$, we employ much smaller
search spaces.
This strategy allowed us to exclude some cases very quickly, and pointed to
some open cases where it appears likely that no solutions exist, and hence
indicated that some potential additional prime divisors of $C_n'$ may be
required.
For example, at $n=17$ with $p=71$, this method found that no solutions
exist where $\max C \leq 38$ and $\min D \geq 32$, so it appears likely
that $71$ is a required divisor for symmetric ideal solutions of size $17$.
The last column in Table~\ref{tableReqDivZ} displays five cases where we
obtain evidence of additional required divisors in this way.

\begin{table}[tbp]
\caption{Required divisors for $C_n(\mathbb{Z})$, and additional
required divisors for $C_n'(\mathbb{Z})$.
Boldface values were established using computation.
Entries combine information from \cite{Caley12}, \cite{Caley13}, \cite{FM17},
\cite{Kleiman}, \cite{ReesSmyth}, plus new values determined here, which are
boxed.
The last column lists some additional primes where it appears likely that
these are also required, but we do not have proof.}\label{tableReqDivZ}
\begin{center}
\scalebox{.80}{
\begin{tabular}{|c|c|c|c|}\hline
\TS$n$ & Divisors of $C_n$ & Additional for $C_n'$ & Conjec.\\\hline
\TS 3 & $2^2$ & &\\
4 & $2^2 \cdot 3^2$ & &\\
5 & $2^4 \cdot 3^2 \cdot 5 \cdot 7$ & &\\
6 & $2^5 \cdot 3^2 \cdot 5^2$ & &\\
7 & $2^6 \cdot 3^3 \cdot 5^2 \cdot 7 \cdot \mathbf{11}$ & $\mathbf{19}$ &\\
8 & $2^6 \cdot 3^3 \cdot 5^2 \cdot 7^2 \cdot \mathbf{11 \cdot 13}$
& &\\
9 & $2^9 \cdot 3^3 \cdot 5^2 \cdot 7^2 \cdot 11 \cdot \mathbf{13}$ & &\\
10 & $2^7 \cdot 3^4 \cdot 5^2 \cdot 7^2 \cdot 13 \cdot \mathbf{17 \cdot}\;
\boxed{\mathbf{23}}$ & &\\
11 & $2^8 \cdot 3^4 \cdot 5^3 \cdot 7^2 \cdot 11 \cdot 13 \cdot \mathbf{17
\cdot 19}$ & $\mathbf{31}$ &\\
12 & $2^8 \cdot 3^4 \cdot 5^3 \cdot 7^2 \cdot 11^2 \cdot \mathbf{17 \cdot
19}$ & $\mathbf{29}$ &\\
13 & $2^{10} \cdot 3^5 \cdot 5^3 \cdot 7^2 \cdot 11^2 \cdot 13 \cdot
\boxed{\mathbf{17 \cdot 19 \cdot 23}}$ & $\mathbf{29 \cdot 31 \cdot
37 \cdot 41}$ &\\
14 & $2^{10} \cdot 3^5 \cdot 5^3 \cdot 7^2 \cdot 11^2 \cdot 13^2 \cdot 17
\cdot \boxed{\mathbf{19 \cdot 23}}$ & $\boxed{\mathbf{31 \cdot 37}}$ &\\
15 & $2^{11} \cdot 3^5 \cdot 5^3 \cdot 7^3 \cdot 11^2 \cdot 13^2 \cdot 17
\cdot \boxed{\mathbf{19 \cdot 23}}$ & $\boxed{\mathbf{37 \cdot 41 \cdot 43
\cdot 47}}$ &\\
16 & $2^{11} \cdot 3^6 \cdot 5^4 \cdot 7^3 \cdot 11^2 \cdot 13^2 \cdot 19
\cdot \boxed{\mathbf{23}}$ & $\boxed{\mathbf{29 \cdot 37 \cdot 41 \cdot 43
\cdot 53}}$ &\\
17 & $2^{15} \cdot 3^6 \cdot 5^4 \cdot 7^3 \cdot 11^2 \cdot 13^2 \cdot 17
\cdot 19 \cdot \boxed{\mathbf{23 \cdot 29}}$ & $\boxed{\mathbf{31 \cdot 37
\cdot 41 \cdot 43 \cdot 47 \cdot 53}}$ & $\mathbf{71}$\\
18 & $2^{15} \cdot 3^6 \cdot 5^4 \cdot 7^3 \cdot 11^2 \cdot 13^2 \cdot 17^2
\cdot \boxed{\mathbf{23 \cdot 29}}$ & $\boxed{\mathbf{31 \cdot 41 \cdot 43
\cdot 47 \cdot 59}}$ &\\
19 & $2^{16} \cdot 3^8 \cdot 5^4 \cdot 7^3 \cdot 11^2 \cdot 13^2 \cdot 17^2
\cdot 19 \cdot 23 \cdot \boxed{\mathbf{29 \cdot 31}}$ & $\boxed{\mathbf{41
\cdot 43 \cdot 47 \cdot 53 \cdot 59}}$\ & $\mathbf{61 \cdot 67}$\\
\BS 20 & $2^{16} \cdot 3^8 \cdot 5^4 \cdot 7^3 \cdot 11^2 \cdot 13^2 \cdot 17^2
\cdot 19^2 \cdot 23 \cdot \boxed{\mathbf{29 \cdot 31}}$ & $\boxed{\mathbf{37
\cdot 43 \cdot 47 \cdot 53}}$ & $\mathbf{67 \cdot 71}$\\\hline
\end{tabular}}
\end{center}
\end{table}

\section{Local Solutions}\label{secLocal}

Table~\ref{tableReqDivZ} shows that small primes are occasionally skipped
in the list of required prime divisors in the \PTE\ constant $C_n$.
For example, $p=11$ and $p=19$ are conspicuously missing at $n=10$, and
$p=29$ is omitted at $n=14$.
In fact, we see that $p$ never arises as a required factor for $C_n$ in
this table if $p>n$ and $p\equiv \pm1\bmod n$.
We prove this is always the case by constructing symmetric ideal solutions
in the \PTE\ problem in the local case when $p$ satisfies this constraint.
These then provide obstructions to particular divisibility requirements in
the \PTE\ constant.
Some of these constructions involve local analogues of the Chebyshev
polynomials, so we first require some facts regarding these polynomials.

Recall the Chebyshev polynomial of the first kind $T_n(x) \in \mathbb{Z}[x]$ is
defined by
\begin{equation}\label{eqnChebDef}
T_n(\cos(\theta)) = \cos(n\theta).
\end{equation}
A close relative of this polynomial, often appearing in the setting of finite
fields, is the \textit{Dickson polynomial}, defined by
\begin{equation}\label{eqnDickson}
D_n(x) = \sum_k (-1)^k \frac{n}{n-k} \binom{n-k}{k} x^{n-2k}.
\end{equation}
These polynomials satisfy
\begin{equation}\label{eqnDicksonProp}
D_n(x + x^{-1}) = x^n + x^{-n},
\end{equation}
providing an analogue of \eqref{eqnChebDef}.
These two polynomials are related via $2T_n(x) = D_n(2x)$, and it is more
convenient at times to employ the latter notation.
(We remark that the Dickson polynomials are more generally defined with an
additional parameter $\alpha$: one changes the $(-1)^k$ term in
\eqref{eqnDickson} to $(-\alpha)^k$.
We require only $\alpha = 1$ here so we omit this parameter.)

The following two lemmas establish that certain constant shifts of
Chebyshev polynomials split completely in particular settings.

\begin{lemma}\label{lemChebSplit1}
Let $T_n(x)$ denote the $n$th Chebyshev polynomial, and let $p \equiv \pm1
\bmod n$.
Then $T_n(x)-1$ splits completely in $\mathbb{F}_p[x]$.
\end{lemma}

\begin{proof}
Suppose first $p \equiv 1\bmod n$.
Write $p = rn + 1$.
Let $a$ be a primitive root mod $p$, and let $b = a^r$.
Since
\[
D_n(b^k + 1/b^k) = 2,
\]
we have that $T_n((b^k+1/b^k)/2) = 1$ for $0\leq k < n$.
This specifies $1 + (n-1)/2$ distinct roots.
We claim all but $x = 1$ (that is, the cases $ 1\leq k < n$) are in fact
double roots.
For this, differentiate \eqref{eqnDicksonProp} to obtain
\[
(x^2-1) D_n'(x) = n(x^n - 1/x^n).
\]
(This is essentially the fact that $T_n'(x) = nU_{n-1}(x)$, where
$U_{n-1}(x)$ is the Chebyshev polynomial of the second kind; recall
$U_{n-1}(\cos(t)) = \sin(nt)/\sin(t)$.)
Thus each $b^k$ with $1\leq k<n$ is a root of $D_n'(x)$: note $-1$ cannot
be among our original solutions.
Thus, $D_n(x)-2$ splits completely mod $p$.

Next, suppose $p \equiv -1\bmod n$.
Write $p = rn - 1$.
Let $a$ be a primitive element of $\mathbb{F}_{p^2}^*$, and let $b =
a^{r(p-1)}$.
Then
\[
D_n(b^k + 1/b^k) = 2
\]
for each $k$, and so $T_n((b^k+1/b^k)/2) = 1$ for $0\leq k < n$.
Now $b^k + 1/b^k$ is a member of the base field $\mathbb{F}_p$ so again we
have $(n+1)/2$ distinct roots of $T_n(x) - 1$ mod $p$.
The fact that $b^k + 1/b^k$ is a double root for $1\leq k \leq (n-1)/2$
follows as above.
\end{proof}

\begin{lemma}\label{lemChebSplit2}
Let $T_n(x)$ denote the $n$th Chebyshev polynomial, let $p \equiv \pm1 \bmod n$,
and suppose $x_0^n + 1/x_0^n = 2c$ for some $x_0\in\mathbb{F}_p^*$ and 
$c\in\mathbb{F}_p$.
Then $T_n(x) - c$ splits completely $\bmod p$.
Furthermore, if $p > n+1$ then at least two such values of $c$ exist.
\end{lemma}

\begin{proof}
Let $b$ denote a primitive $n$th root of unity mod $p$.
If $x_0^n + 1/x_0^n = 2c$, then let
\[
x_k = \pm2^{-1}\left(x_0 b^k + (x_0 b^k)^{-1}\right)
\]
for $0\leq k < n$.
If $x_i = x_j$ for some $i < j$, then one finds that $x_0$ is a $2n$th root
of $1$, so $c$ is either $1$ or $-1$.
The case $c = 1$ was treated by Lemma~\ref{lemChebSplit1}, and there we
found that all the roots except $1$ and $-1$ have multiplicity $2$.
For $c = -1$, the same argument produces that all the roots have
multiplicity $2$.
If $c$ is neither $1$ nor $-1$, then we have $n$ distinct
roots of $T_n(x) - c$ mod $p$, and this polynomial has degree $n$, so it
splits completely mod $p$.

Suppose in addition that $p>n+1$.
We need to ensure that $x_0^n+1/x_0^n$ achieves more than one value $2c$ as
$x_0$ ranges over the nonzero elements of the field in question.
This is clear for $p>2n+1$, since for a fixed $c$ there are at most $2n$
solutions mod $p$ to $x^n + 1/x^n = 2c$.
For $p=2n\pm1$, it is straightforward to verify that $c=1$ and $c=-1$ each
have $(p-1)/2$ distinct solutions mod $p$.
Thus $p>n+1$ suffices.
\end{proof}

We remark that the condition $p>n+1$ is also necessary in the last
statement: when $p=n+1$, the only possibility is $c=1$.

We remark also that we can obtain a little more in this statement: write $p
= rn \pm 1$ as above.
If $r$ is even, then $c = -1$ is allowed, and this along with $c = 1$
each consume $n$ of the nonzero field elements, and the remaining elements
partition into blocks of size $2n$, so that makes $r/2 + 1$ different
allowable values of $c$.
If $r$ is odd, then the case $c = 1$ consumes $n$ elements, and the rest
partition into sets of size $2n$, so there are $(r+1)/2$ distinct allowable
values of $c$.

We now record some explicit constructions for symmetric ideal solutions in
the \PTE\ problem in certain local cases.
These results are contained in two theorems, one for the case of odd $n$,
and another for the case of even $n$.

\begin{theorem}\label{thmOddLocal}
Suppose $n\geq3$ is odd and $p$ is a prime.
\begin{enumerate}[(i)]
\item\label{enumOddA}
If $p\equiv1 \bmod n$ and $a\in\mathbb{F}_p$ has order $n \bmod p$, then
$A=\{1, a, a^2, \ldots, a^{n-1}\}$, $B=-A$ comprises a symmetric ideal
\PTE\ solution of size $n$ in $\mathbb{F}_p$.
\item\label{enumOddB}
If $p = \pm1 \bmod n$, then setting $A$ to the multiset of roots of
$T_n(x)-1$ over $\mathbb{F}_p$ and $B=-A$ comprises a symmetric ideal \PTE\
solution of size $n$ in $\mathbb{F}_p$.
\end{enumerate}
\end{theorem}

\begin{proof}
For (\ref{enumOddA}), we note that if $n\nmid m$ then $a^m\not\equiv 1\bmod p$ and
\[
(a^m-1)\sum_{k=0}^{n-1} a^{mk} = a^{mn}-1 \equiv 0 \pmod p,
\]
so all of the required even moments of $A$ and $B$ are $0$.
For (\ref{enumOddB}), from Lemma~\ref{lemChebSplit1} we know that $A$ has $n$
elements, with each element except $1$ appearing twice.
Write $A = \{1\} \cup A' \cup A'$ (as a multiset), where $A' = \{(b^k +
b^{-k})/2 : 1\leq k \leq (n-1)/2\}$, with $b$ having order $n$ (mod $p$) as in
Lemma~\ref{lemChebSplit1}.
If $a \in A$ then $T_n(a) = 1$ so $T_n(-a) = -1$ since $T_n(x)$ is an odd
function when $n$ is odd, so $-a \not\in A$.
Also, since $T_n(x)$ is odd of degree $n\geq3$ all the odd moments $m\leq n-2$
of $A$ and $B$ are $0$.
\end{proof}

For example, when $n=9$ and $p=19$, part (\ref{enumOddA}) with $a=4$
produces the solution
\begin{equation}\label{eqn9p19i}
A_1 = \{1,4,5,6,7,9,11,16,17\},
\end{equation}
and part (\ref{enumOddB}) yields
\begin{equation}\label{eqn9p19ii}
A_2 = \{1,9,9,11,11,13,13,14,14\}.
\end{equation}
Indeed, we verified these are the only two symmetric ideal solutions of
size $9$ mod $19$, so if $19$ does not divide the constant in a \PTE\
solution over $\mathbb{Z}$, then it must be congruent mod $19$ to a unit
multiple of one of these two solutions.
We remark that the two known symmetric ideal solutions \eqref{eqnPTE9} of size
$9$ over $\mathbb{Z}$ do not have $19$ as a factor of their respective
constants.
The smaller one is congruent to $12A_2$, the larger, to $9A_1$.
Similarly, these constructions produce the only two local solutions at
$n=11$ with $p=23$: $\{1,2,3,4,6,8,9,12,13,16,18\}$ and
$\{1,2,2,3,3,5,5,7,7,17,17\}$.

In general, these constructions do not describe all local solutions.
For example, size $n=9$ admits the local solution
\begin{equation}\label{eqn9p23}
\{1,1,1,3,8,11,11,16,17\}
\end{equation}
at $p=23$, and $n=11$ has
\begin{equation}\label{eqn11p29}
\{1,1,1,15,19,21,21,22,23,25,25\}
\end{equation}
at $p=29$.
These are in fact the unique solutions for these cases, up to unit
multiples.

We next provide similar constructions for the even case.

\begin{theorem}\label{thmEvenLocal}
Suppose $n\geq2$ is even and $p$ is a prime.
\begin{enumerate}[(i)]
\item\label{enumEvenA}
If $p\equiv \pm1\bmod n$ and $p>n+1$, and $c_1 \neq c_2$ have the property
that $T_n(x)-c_1$ and $T_n(x)-c_2$ both split over $\mathbb{F}_p$, then
setting $A$ and $B$ to the multiset of roots of these repective polynomials
over $\mathbb{F}_p$ comprises a symmetric ideal \PTE\ solution of size $n$ in
$\mathbb{F}_p$.
\item\label{enumEvenB}
If $p \equiv 1\bmod n$, let $a$ denote an element of order $n \bmod p$.
Then there exist $r:=(p-1)/n$ choices for an integer $s\in[0,p-1)$,
including $s=0$, such that $A = \{1, a, \ldots, a^{n-1}\}$ and $B = sA$
comprises a symmetric ideal \PTE\ solution of size $n$ in $\mathbb{F}_p$.
\end{enumerate}
\end{theorem}

\begin{proof}
Part (\ref{enumEvenA}) follows directly from Lemma~\ref{lemChebSplit2}.
For (\ref{enumEvenB}), let $1 \leq m < n/2$.
Since $a^{2m} \not\equiv 1\bmod p$, we have
\[
(a^{2m}-1)\sum_{k=0}^{n/2-1} a^{2mk} = a^{mn} - 1 \equiv 0 \pmod p,
\]
so all the required even moments of $A$ and $B$ are $0$ mod $p$.
We need only ensure that $A$ does not intersect $\pm B$.
This is certainly the case at $s=0$.
The fact that there are $r$ different choices for $s$ follows from the fact
that the set of $r$th powers mod $p$ forms a subgroup of index $r$ in the
multiplicative group $\mathbb{F}_p^*$.
\end{proof}

For example, at $n=10$ construction (\ref{enumEvenA}) produces
$A=\pm\{0,3,3,4,4\}$, $B=\pm\{1,2,2,7,7\}$ using $c=-1$ and $c=1$
respectively when $p=19$, and $A=\pm\{1,3,3,12,12\}$,
$B=\pm\{2,4,7,10,14\}$ using $c=1$ and $c=14$ when $p=29$.
Also at $n=10$, construction (\ref{enumEvenB}) with $a=2$ produces
$A=\pm\{1,2,3,4,5\}$, $B=\{0,0,0,0,0,0,0,0,0,0\}$ with $p=11$ and $s=0$ (the
only local solution up to scalar multiplication for this case), and at $p=31$
it yields the three solutions $A=\pm\{1,2,4,8,16\}$ and $B=sA$ with $s=0$,
$3$, or $5$.

\section[Computations in $\mathbb{Z}$]{Computations in \protect{\boldmath{$\mathbb{Z}$}}}\label{secCompZ}

We describe searches for symmetric ideal solutions in the \PTE\ problem
over $\mathbb{Z}$ with size $n$ satisfying $9\leq n\leq16$ and height
bounded by a given integer $H=H(n)$.
Our method builds on the work of Borwein et al.~\cite{BLP}, but adds some significant new
features.
We describe our strategy for odd $n$ first, then for even $n$.

\subsection[Searches With $n$ Odd]{Searches With \protect{\boldmath{$n$}} Odd}\label{subsecOddSearches}

Given a positive integer $H$ and an odd positive integer $n$, we wish to
determine all symmetric ideal solutions $A=\{a_1, \ldots, a_n\}$, $B=-A$,
with $\abs{a_i}\leq H$ for each $i$.
In our method, the prime divisors of the \PTE\ constant $C_n'$ have a central
role.

If a prime $p$ divides the constant $C_n(A,B)$ for an ideal \PTE\ solution
$(A,B)$, then from \eqref{eqnCnZ} and the fact that $\mathbb{F}_p[x]$ is a
\UFD\ it follows that $A\equiv B$ mod $p$ as multisets.
Thus, if $A$ is a symmetric ideal solution of size $n$ with $n$ odd, then
necessarily $A=-A$ mod $p$ for each $p\mid C'_n$.
Further, since $n$ is odd, each such prime $p$ must divide an odd number of
elements in any solution $A$.
Borwein et al.\ \cite{BLP} used this observation to trim their search for
symmetric ideal \PTE\ solutions over the integers with odd size $n$.
Let $p_1>p_2$ denote the two largest primes required in $C_n'$ from
Table~\ref{tableReqDivZ}, so at $n=11$ we have $p_1=31$ and $p_2=19$.
We may assume that $p_1\mid a_1$.
If $p_2\nmid a_1$, we can select $a_2\in[-H,H]$ so that $p_2\mid
(a_1+a_2)$.
After this, we can enforce $p_1\mid(a_2+a_3)$ when selecting $a_3$, unless
$p_1$ already divides $a_1+a_2$.
In general, we require
\begin{equation}\label{eqnTwoPrimes}
\begin{split}
a_1 &\equiv 0 \pmod{p_1},\\
(a_1+a_2)a_1 &\equiv 0 \pmod{p_2},\\
(a_2+a_3)(a_1+a_2) &\equiv 0 \pmod{p_1},\\
(a_k+a_{k+1})\sum_{j=1}^k a_j &\equiv 0 \pmod{p_{\pi(k)}},\\
\end{split}
\end{equation}
where $\pi(k)=1$ if $k$ is even and $2$ if $k$ is odd.
In \cite{BLP}, this procedure was employed to select $\ceil{n/2}$ elements.
After this, polynomial interpolation was used to determine if it is
possible to extend this to a symmetric ideal solution of size $n$.
This stems from the observation that in a symmetric solution $(A,-A)$ of
size $n$, Equation~\eqref{eqnCnZ} implies that
\[
\prod_{i=1}^n(a_i+a_j) = C_n(A,-A)
\]
for any $j$, so
\[
C_n(A,-A)^{-1} \prod_{i=\ceil{n/2}+1}^{n} (a_i+a_j) = \prod_{i=1}^{\ceil{n/2}} (a_i+a_j)^{-1}
\]
for $1\leq j\leq\ceil{n/2}$.
If $a_1, \ldots, a_{\ceil{n/2}}$ are distinct, and none is the negation
of another, then there is a unique interpolating polynomial $g(x)$ having
degree $\floor{n/2}$ that satisfies
\begin{equation}\label{eqnPolyInterp1}
g(a_j) = \prod_{i=1}^{\ceil{n/2}} (a_i+a_j)^{-1}
\end{equation}
for $1\leq j\leq\ceil{n/2}$.
Thus, given a qualifying list $a_1, \ldots, a_{\ceil{n/2}}$, one may
construct this polynomial $g(x)$, and then test if all of its roots are
integers.

We employ the strategy of \eqref{eqnTwoPrimes} to select the first
$\ceil{n/2}$ values, but then depart from \cite{BLP} by turning to other
tests, rather than building an interpolating polynomial at this stage.
First, after selecting $a_1, \ldots, a_{\ceil{n/2}}$, we check if this
list can possibly be a subset of a symmetric ideal solution of size $n$
modulo two or three additional required prime divisors $q_i$ of $C_n'$, by
checking if it is possible to augment the current selections to construct an
$A$ with the required property $A\equiv A'\bmod q_i$.
For example, at $n=11$ we might use $q_1=17$, $q_2=13$, and $q_3=11$.
Any set that fails one of these checks is discarded at this point.
After this, we select an additional integer $a_{\ceil{n/2}+1}$ using
\eqref{eqnTwoPrimes}, and check the local conditions modulo the $q_i$
again.
For sets passing these conditions, we employ a particular polynomial
$F_n(x_1,\ldots,x_{\floor{n/2}})$, depending only on $n$, which allows us
to test very quickly if a set of $\ceil{n/2}+1$ selections for the $a_i$
can be extended to form a symmetric ideal solution of size~$n$.

We describe the construction of the polynomial $F_n$.
For $1\leq k\leq\floor{n/2}$, define the \textit{weight} of the variable
$x_k$ by $W(x_k)=2k-1$, and extend this function to monomials by letting
the weight act additively: $W(x_i x_j)=W(x_i)+W(x_j)$.
Let $D$ be a positive integer whose value will be selected later, and let
$\mathcal{M}_{D,n}$ denote the monomials in $\floor{n/2}$ variables having
weight $D$:
\[
\mathcal{M}_{D,n} = \left\{\prod_{k=1}^{\floor{n/2}} x_k^{d_k} :
d_k\geq0 \mathrm{\ and\ } \sum_{k=1}^{\floor{n/2}} d_k W(x_k) = D\right\}.
\]
For example, when $n=9$ and $D=10$ this set has nine elements:
\[
\mathcal{M}_{10,9} = \{x_1^{10},\, x_1^7 x_2,\, x_1^4 x_2^2,\, x_1 x_2^3,\,
x_1^5 x_3,\, x_1^2 x_2 x_3,\, x_3^2,\, x_1^3 x_4,\, x_2 x_4\}.
\]
Note that the parity of the total degree of each monomial matches the
value of $D$ mod $2$.

Consider the set of polynomials obtained from
$\mathcal{M}_{D,n}$ by replacing each variable $x_k$ with the
polynomial $\sum_{j=(n+5)/2}^n a_j^{2k-1}$, that is, the $(2k-1)$th
moment of the $(n-3)/2$ variables corresponding the elements of our
set that have not yet been selected.  If $D$ is sufficiently large,
then we may use linear algebra to determine a nontrivial linear
combination of these polynomials which is identically $0$.  Applying
this linear combination to the elements of $\mathcal{M}_{D,n}$
produces our polynomial $F_n$.

If a list $a_1, \ldots, a_{(n+3)/2}$ can be extended to a symmetric ideal
solution $A$, then each of the odd moments $m=1$, $3, \ldots, n-2$ of $A$
must be $0$, so
\begin{equation}\label{eqnPowerSplit}
-\sum_{j=1}^{\frac{n+3}{2}} a_j^{2k-1} = \sum_{j=\frac{n+5}{2}}^n a_j^{2k-1}
\end{equation}
for $1\leq k\leq (n-1)/2$.  Consequently (and using the fact that the
degrees of the monomials in $\mathcal{M}_{D,n}$ all have the same
parity), the polynomial $F_n$ evaluated at these odd moments of the
selected integers $a_1, \ldots, a_{(n+3)/2}$ must be $0$.  This is a
necessary condition for a solution to exist, though it is not
sufficient.

When $n=9$, selecting $D=10$ suffices, and we compute
\begin{equation*}
\begin{split}
F_9(x_1,x_2,x_3,x_4) &= 
x_1^{10} - 15 x_1^7 x_2 - 175 x_1 x_2^3 + 63 x_1^5 x_3 + 315 x_1^2 x_2 x_3
- 189 x_3^2\\
&\qquad{} - 225 x_1^3 x_4 + 225 x_2 x_4.
\end{split}
\end{equation*}
When $n=11$, we require $D=15$, and find
\begin{equation*}
{\small
\begin{split}
&F_{11}(x_1,x_2,x_3,x_4,x_5) = 
x_1^{15}
- 35 x_1^{12} x_2
+ 175 x_1^9 x_2^2
- 1225 x_1^6 x_2^3
- 12250 x_1^3 x_2^4
+ 6125 x_2^5\\
&+ 252 x_1^{10} x_3
+ 945 x_1^7 x_2 x_3
+ 33075 x_1^4 x_2^2 x_3
- 11025 x_1 x_2^3 x_3
- 11907 x_1^5 x_3^2
- 59535 x_1^2 x_2 x_3^2\\
&- 35721 x_3^3
- 2025 x_1^8 x_4 
- 14175 x_1^5 x_2 x_4
+ 70875 x_1^2 x_2^2 x_4
+ 42525 x_1^3 x_3 x_4
+ 85050 x_2 x_3 x_4\\
&- 91125 x_1 x_4^2
+ 11025 x_1^6 x_5
- 55125 x_1^3 x_2 x_5
- 55125 x_2^2 x_5
+ 99225 x_1 x_3 x_5.
\end{split}}\normalsize
\end{equation*}
At $n=13$, we need $D=21$, and compute
\begin{equation*}
{\tiny
\begin{split}
&F_{13}(x_1,x_2,x_3,x_4,x_5,x_6) = 
x_1^{21}
- 70 x_1^{18} x_2
+ 1155 x_1^{15} x_2^2
- 9800 x_1^{12} x_2^3
- 67375 x_1^9 x_2^4
- 1414875 x_1^6 x_2^5
\\ &\quad
+ 4716250 x_1^3 x_2^6
+ 2358125 x_2^7
+ 756 x_1^{16} x_3
- 6615 x_1^{13} x_2 x_3
+ 291060 x_1^{10} x_2^2 x_3
+ 4729725 x_1^7 x_2^3 x_3
\\ &\quad
- 12733875 x_1^4 x_2^4 x_3
- 12733875 x_1 x_2^5 x_3
- 83349 x_1^{11} x_3^2
- 3929310 x_1^8 x_2 x_3^2
- 13752585 x_1^5 x_2^2 x_3^2
+ 45841950 x_1^2 x_2^3 x_3^2
\\ &\quad
+ 8251551 x_1^6 x_3^3
- 41257755 x_1^3 x_2 x_3^3
- 41257755 x_2^2 x_3^3
- 74263959 x_1 x_3^4
- 10125 x_1^{14} x_4
- 56700 x_1^{11} x_2 x_4
\\ &\quad
- 2338875 x_1^8 x_2^2 x_4
+ 27286875 x_1^5 x_2^3 x_4
- 27286875 x_1^2 x_2^4 x_4
+ 2338875 x_1^9 x_3 x_4
+ 9823275 x_1^6 x_2 x_3 x_4
\\ &\quad
+ 49116375 x_1^3 x_2^2 x_3 x_4
+ 81860625 x_2^3 x_3 x_4
+ 265228425 x_1 x_2 x_3^2 x_4
- 6014250 x_1^7 x_4^2
- 105249375 x_1^4 x_2 x_4^2
\\ &\quad
- 105249375 x_1 x_2^2 x_4^2
- 189448875 x_1^2 x_3 x_4^2
+ 135320625 x_4^3
+ 121275 x_1^{12} x_5
+ 606375 x_1^9 x_2 x_5
- 25467750 x_1^6 x_2^2 x_5
\\ &\quad
- 21223125 x_1^3 x_2^3 x_5
- 42446250 x_2^4 x_5
- 3274425 x_1^7 x_3 x_5
+ 114604875 x_1^4 x_2 x_3 x_5
- 229209750 x_1 x_2^2 x_3 x_5
\\ &\quad
+ 206288775 x_1^2 x_3^2 x_5
+ 49116375 x_1^5 x_4 x_5
+ 245581875 x_1^2 x_2 x_4 x_5
- 294698250 x_3 x_4 x_5
- 191008125 x_1^3 x_5^2
\\ &\quad
+ 191008125 x_2 x_5^2
- 893025 x_1^{10} x_6
+ 13395375 x_1^7 x_2 x_6
+ 156279375 x_1 x_2^3 x_6
- 56260575 x_1^5 x_3 x_6
\\ &\quad
- 281302875 x_1^2 x_2 x_3 x_6
+ 168781725 x_3^2 x_6
+ 200930625 x_1^3 x_4 x_6
- 200930625 x_2 x_4 x_6.
\end{split}}\normalsize
\end{equation*}
At $n=15$, we require $D=28$.
The polynomial $F_{15}$ has $159$ terms and height (i.e., maximal
coefficient in absolute value) $18101840006250$.

The polynomial $F_n$ has a central role in the remaining
steps of the method.  It is first employed in an additional local
check modulo another prime $r$, which is different from any of the
$p_i$ or $q_i$ ($r$ need not be a divisor of $C_n'$).  At the start of
the program, we precompute the odd moments $m_1, \ldots, m_{n-2}$ of
all $((n-3)/2)$-tuples over $\mathbb{Z}/r\mathbb{Z}$ and record the
negation of their values in $(\mathbb{Z}/r\mathbb{Z})^{(n-1)/2}$ that
are achieved in a table $T_{n,r}$.  We use this table together with
the polynomial $F_n$ to restrict the choices of $a_{\ceil{n/2}+1}$ mod
$r$.

We illustrate this using the case $n=11$, where we use $r=23$.
After $a_1, \ldots, a_6$ have been selected, and have passed the tests
mod $q_1$, $q_2$, and $q_3$, we compute all possible values of $a_7$ mod
$r$ so that $F_{11}(m_1,m_3,m_5,m_7,m_9)\equiv0 \bmod r$, where $m_i$
denotes the $i$th moment of $\{a_1,\ldots,a_7\}$.
Then we can easily determine if a particular choice of $a_7$ mod $r$ is
plausible, given $a_1, \ldots, a_6$, by checking if the five moments of
these seven values match an entry in the table $T_{n,r}$.
This accelerates the inner loop significantly, since we need only consider
approximately $2H/p_1 r$ values for each valid congruence class for $a_7$.

After this, for each integer $a_7\in[-H,H]$ congruent to one of these
permissible values mod $r=23$, and in an allowed congruence class mod
$p_1=31$ according to the strategy \eqref{eqnTwoPrimes}, and for which
$\{a_1,\ldots,a_7\}$ remains a viable subset modulo $q_1$, $q_2$, and
$q_3$, we test if $F_{11}(m_1,m_3,m_5,m_7,m_9)=0$.
However, rather than compute this value in $\mathbb{Z}$, it is more
efficient to compute it modulo $M_i$ for a few integers $M_i$.
We use $M_1=2^{50}$ first, as using a power of $2$ is helpful for speed.
If the value is $0$ here, then we use a second modulus, the prime
$M_2=2^{30}+3$.
The few sets that survive both of these tests we then check over
$\mathbb{Z}$ using Magma \cite{Magma} (version 2.25-3).
If this is verified as well then we use Magma again to compute a Gr\"obner
basis to determine if the polynomial system \eqref{eqnPowerSplit} has a
solution for $a_8$, $a_9$, $a_{10}$, $a_{11}$, given $a_1, \ldots, a_7$.

\subsubsection*{Results}
We implemented this method in C and used it to search for
symmetric ideal solutions in the \PTE\ problem at odd sizes $n=9$,
$11$, $13$, and $15$.
For each $n$, Table~\ref{tableZComp} displays the height bound $H$ employed in
our calculations, the bound $H_0$ reached in \cite{BLP} (if applicable), and
the values of the parameters $p_i$, $q_i$, and $r$ chosen in our calculations.
This table also summarizes our results.
For each $n$, it displays the number of primitive, symmetric, ideal solutions
at size $n$ with height at most $H$.
(A solution is \textit{primitive} if the members of $A$ and $B$ have no
nontrivial common factor.)
Some additional details on the results at each size are given below.

\subsubsection*{$n=9$:}
Using $(p_1,p_2)=(13,11)$ as in Table~\ref{tableReqDivZ}, we
found there are no primitive integral symmetric solutions up to height
$4000$ outside of those discovered by Letac \eqref{eqnPTE9}.
This search also reported several thousand examples which had at least five
integers and with algebraic irrationals as the remaining values, including
more than $100$ consisting of seven integers and two quadratic irrationals.
(There are likely many more such examples with height at most $4000$; the
ones we found are merely those that survived all our necessary conditions
for integer solutions.)
Some examples with small height are
\begin{equation*}
\begin{split}
A_1 &= \{-15, -13, 4-\sqrt{14}, -7, 2, 3, 10, 12, 4+3\sqrt{14}\},\\
A_2 &= \{-18, 2-\sqrt{259}, -10, -7, 1, 5, 12, 13, 2+\sqrt{259}\},\\
A_3 &= \{-30, (5-\sqrt{345})/2, -16, -14, 2, 9, 21, 23, (5+3\sqrt{345})/2\},\\
A_4 &= \{-33, -29, -25, -6, 18-\sqrt{231}, 3, 26, 28, 18+\sqrt{231}\}.
\end{split}
\end{equation*}
We then extended our search by employing our work on local solutions.
Since \eqref{eqn9p23} is the only symmetric ideal \PTE\ solution up to unit
multiples at $n=9$ in $\mathbb{Z}/23\mathbb{Z}$, we can first search for
integral solutions which reduce mod $23$ to a unit multiple of this
multiset, and then perform our search using $23$ as a required divisor in
$C_9'$.
Likewise, \eqref{eqn9p19i} and \eqref{eqn9p19ii} are the only obstructions
at the prime $19$ up to unit multiples, and we can check these separately.
Performing these tests allowed us to use $(p_1,p_2)=(23,19)$ in our method,
which enabled us to check to height $7000$.
We found that no new primitive symmetric ideal solutions for $n=9$ exist
with this height bound.

\subsubsection*{$n=11$:}
We first look for symmetric ideal solutions that reduce mod $29$
to a unit multiple of the lone local solution \eqref{eqn11p29} at this
prime, and then we use $(p_1,p_2)=(31,29)$ to search for integer solutions
at this size with height at most $3500$.
We find there are no symmetric ideal solutions for $n=11$ with this bound.
However, we found several hundred solutions where
$F_{11}(m_1,m_3,m_5,m_7,m_9)=0$ and the resulting polynomial system
contained irrational solutions.
In two of these, in fact, the solution consisted of nine integers and two
real quadratic irrationals:
\begin{equation}\label{eqnNear11}
\begin{split}
A_1 &= \{-95, -68-\alpha, -52, -48, -13, -9, 30, 34, 61, 65, 95+\alpha\},\\
A_2 &= \{-589, -546+\beta, -363, -185, -170, 41, 87-\beta, 234, 355, 548, 588\}
\end{split}
\end{equation}
(and $B_i=-A_i$ in each case), where $\alpha=0.0490649\ldots$ is the
positive real root of $x^2+163x-8$, and $\beta=0.0473969\ldots$ is the
smaller real root of $x^2-633x+30$.
(Both of these examples arose in preliminary searches to a smaller height
using $p_2=19$.)
We thus find symmetric ideal solutions to the \PTE\ problem at $n=11$ in
some real quadratic extensions of $\mathbb{Q}$.

\subsubsection*{$n=13$:}
Using $p_1=41$ and $p_2=37$ from Table~\ref{tableReqDivZ} we
find that no solutions in integers exist with height at most $2000$.
Furthermore, just two sets $\{a_1,\ldots,a_8\}$ were found that admit
some algebraic irrational members, totally real algebraic integers of
degree $5$ in both cases:
\begin{align*}
A_1 &= \{-1026, -532, -248, -245, 207, 286, 533, 1025\} \cup \{\alpha : f_1(\alpha)=0\},\\
A_2 &= \{-1827, -989, -715, -405, 64, 902, 1330, 1640\} \cup \{\alpha : f_2(\alpha)=0\},
\end{align*}
where
\begin{align*}\label{eqnNear9}
f_1(x) &= x^5 - 1215245 x^3 + 170781322924 x + 931022099280,\\
f_2(x) &= x^5 - 4278125 x^3 + 3664306448884 x - 975527373189600.
\end{align*}

\subsubsection*{$n=15$:}
We find no integral solutions exist for $n=15$ with height at most $1100$.
A few solutions with nine integers and the roots of a sextic arise,
including
\begin{equation*}
\{-340, -246, -235, -152, -141, 188, 199, 293, 387\} \cup \{\alpha : f(\alpha)=0\},
\end{equation*}
where
\begin{align*}
f(x) &= x^6 - 47x^5 - 221471x^4 + 18495675x^3 + 11686424346x^2\\
&\qquad - 1340204033304x + 2551344634800.
\end{align*}

\begin{table}[tp]
\caption{Summary of searches.
$H_0$ is the height bound employed in \cite{BLP}; $H$ is the height bound
reached in this work.
An asterisk on a prime $p_i$ indicates that a separate local search with
this prime was performed to justify its use.
The last line shows the number of distinct primitive, symmetric, ideal
\PTE\ solutions over $\mathbb{Z}$ with height at most $H$.}\label{tableZComp}
\begin{center}
\small
\begin{tabular}{|c|cccccccc|}\hline
$n$ & $9$ & $10$ & $11$ & $12$ & $13$ & $14$ & $15$ & $16$\\\hline
$H_0$ & $2000$ & $1500$ & $2000$ & $1000$ & --- & --- & --- & ---\\
$H$ & $7000$ & $2500$ & $3500$ & $1511$& $2000$ & $900$ & $1100$ & $850$\\
$p_1$ & $23^*$ & $29^*$ & $31$ & $31^*$ & $41$ & $37$ & $47$ & $53$\\
$p_2$ & $19^*$ & $23$ & $29^*$ & $29$ & $37$ & $31$ & $43$ & $43$\\
$q_1$ & $13$ & $17$ & $19$ & $19$ & $31$ & $23$ & $41$ & $41$\\
$q_2$ & $11$ & $13$ & $17$ & $17$ & $29$ & $19$ & $37$ & $37$\\
$q_3$ & --- & --- & $13$ & $11$ & $23$ & $17$ & $23$ & $29$\\
$r$ & $31$ & $31$ & $23$ & $23$ & $19$ & $29$ & $19$ & $23$\\\hline
Solutions & $2$ & $2$ & $0$ & $2$ & $0$ & $0$ & $0$ & $0$\\\hline
\end{tabular}
\end{center}
\end{table}

\subsection[Searches With $n$ Even]{Searches With \protect{\boldmath{$n$}} Even}\label{subsecEvenSearches}

Given a bound $H$ and an even integer $n$, we need to determine all
symmetric ideal solutions $A=\pm\{a_1,\ldots,a_{n/2}\}$,
$B=\pm\{b_1,\ldots,b_{n/2}\}$ with $1\leq a_i\leq H$ and $1\leq b_i\leq H$
for each $i$.
If $p\mid C_n'$ then we require $A\equiv B \bmod p$ as multisets.
Let $p_1>p_2$ be the two largest known prime divisors of $C_n'$.
We select $a_1$ freely in $[1,H]$, and then for $k\geq1$ we require
\begin{align}\label{eqnEvenabSelection}
b_k^2 &\equiv a_k^2 \pmod{p_1},\quad
(a_{k+1}^2 - b_k^2)\sum_{j=1}^k (a_j^2-b_j^2) \equiv 0\pmod{p_2}.
\end{align}
Thus, there are always one or two residue classes mod $p_1$ to check for
$b_k$.
For $a_{k+1}$, in most cases there are one or two classes mod $p_2$ to
check, but if $\sum_{j=1}^k (a_j^2-b_j^2)$ happens to be a multiple of
$p_2$, then $a_{k+1}$ is an unrestricted choice in $[1,H]$.

Borwein et al.\ \cite{BLP} used this strategy to select elements of $A$ and
$B$ until reaching $n/2+1$ total elements, and then they applied a
variation of their polynomial interpolation strategy to check if the selected
elements could be completed to a symmetric ideal solution of size $n$.
We summarize this method here; see \cite{BLP} for details.
If $k$ is any fixed integer in $[1,n/2]$, then
\[
C_n(A,B)^{-1} \prod_{i=n/2-k+2}^{n/2} (b_j^2-a_i^2) =
\prod_{i=1}^{n/2-k+1} (b_j^2-a_i^2)^{-1}
\]
for $1\leq j\leq k$.
Thus, if $a_1, \ldots, a_{n/2-k+1}$ and $b_1, \ldots, b_k$ are
selected (with no repeated values and no element equal to the negation of
another), then one can test the roots of the unique interpolating polynomial
$g_0(x)$ with degree $k-1$ satisfying
\begin{equation}\label{eqnPolyInterp0}
g_0(a_j^2) = \prod_{i=1}^{n/2-k+1} (b_j^2-a_i^2)^{-1}
\end{equation}
for $1\leq j\leq k$.
The authors of \cite{BLP} thus use $k=\floor{n/4}$.

In place of this interpolation strategy, we adapt our method from
Section~\ref{subsecOddSearches} to construct a polynomial $F_n$ in $n/2-1$
variables with the property that $F_n$ vanishes at a point constructed from
the selected values $a_i$ and $b_i$ if these selected values can be
extended to a full symmetric ideal solution at $n$.
Since we are concerned with the even moments, this time we define the
weight of the variable $x_k$ by $W'(x_k)=2k$ and extend it additively on
products as before.
For a positive integer $D$, we let $\mathcal{M}'_{D,n}$ denote the monomials
in $n/2-1$ variables having weight $D$:
\[
\mathcal{M}'_{D,n} = \left\{\prod_{k=1}^{n/2-1} x_k^{d_k} :
d_k\geq0 \mathrm{\ and\ } \sum_{k=1}^{n/2-1} d_k W'(x_k) = D\right\}.
\]

Having selected the first $s=\ceil{n/4}+1$ elements for $A$ and the first
$t=\floor{n/4}+1$ elements for $B$, for a total of $n/2+2$ elements, then
we need the remaining elements $a_{s+1}, \ldots, a_n$, $b_{t+1}, \ldots,
b_n$ to satisfy
\begin{equation}\label{eqnEvenMoments}
\sum_{i=s+1}^n a_i^{2k} - \sum_{i=t+1}^n b_i^{2k} =
\sum_{i=1}^t b_i^{2k} - \sum_{i=1}^s a_i^{2k}
\end{equation}
for $0<2k<n$.
Consider the set of polynomials obtained from $\mathcal{M}'_{D,n}$ by
replacing each variable $x_k$ by the expression on the left side of
\eqref{eqnEvenMoments}.
For sufficiently large $D$, we use linear algebra to determine a nontrivial
linear combination of these polynomials which is identically $0$.
Applying this combination to the elements of $\mathcal{M}'_{D,n}$ produces
our polynomial $F_n$.
This polynomial must evaluate to $0$ if we replace $x_k$ with the
known right side of \eqref{eqnEvenMoments} if the current selection of
elements of $A$ and $B$ can be completed to a full symmetric ideal
solution.

We record the polynomials constructed for some even integers $n$.
When $n=10$, choosing $D=12$ suffices, and we compute
\begin{equation*}
\begin{split}
F_{10}(x_1,x_2,x_3,x_4) &= x_1^6 - 3 x_1^4 x_2 + 9 x_1^2 x_2^2
+ 9 x_2^3 - 8 x_1^3 x_3 - 24 x_1 x_2 x_3 + 16 x_3^2\\
&\qquad{} + 18 x_1^2 x_4 - 18 x_2 x_4.
\end{split}
\end{equation*}
At $n=12$, we require $D=18$, and find
\begin{equation*}
\begin{split}
F_{12}(&x_1,x_2,x_3,x_4,x_5) = x_1^9 + 18 x_1^5 x_2^2 - 135 x_1 x_2^4
- 24 x_1^6 x_3 + 360 x_1^2 x_2^2 x_3 + 320 x_3^3\\
&{}- 360 x_1^3 x_2 x_4 - 720 x_2 x_3 x_4 + 540 x_1 x_4^2 + 144 x_1^4 x_5
+ 432 x_2^2 x_5 - 576 x_1 x_3 x_5.
\end{split}
\end{equation*}
For $n=14$, choosing $D=24$ suffices, and produces
\begin{equation*}
{\small
\begin{split}
&F_{14}(x_1,x_2,x_3,x_4,x_5,x_6) = 
- x_1^{12}
+ 6 x_1^{10} x_2
- 45 x_1^8 x_2^2
+ 60 x_1^6 x_2^3
+ 225 x_1^4 x_2^4
+ 1350 x_1^2 x_2^5
\\ &\quad
- 675 x_2^6
+ 40 x_1^9 x_3
- 720 x_1^5 x_2^2 x_3
- 4800 x_1^3 x_2^3 x_3
+ 1800 x_1 x_2^4 x_3
- 240 x_1^6 x_3^2
+ 3600 x_1^4 x_2 x_3^2
\\ &\quad
+ 3600 x_1^2 x_2^2 x_3^2
- 1200 x_2^3 x_3^2
- 3200 x_1^3 x_3^3
- 9600 x_1 x_2 x_3^3
- 6400 x_3^4
- 90 x_1^8 x_4
\\ &\quad
+ 1080 x_1^6 x_2 x_4
+ 2700 x_1^4 x_2^2 x_4
- 5400 x_1^2 x_2^3 x_4
+ 1350 x_2^4 x_4
- 1440 x_1^5 x_3 x_4
+ 21600 x_1 x_2^2 x_3 x_4
\\ &\quad
+ 7200 x_1^2 x_3^2 x_4
+ 21600 x_2 x_3^2 x_4
- 2700 x_1^4 x_4^2
- 16200 x_1^2 x_2 x_4^2
- 8100 x_2^2 x_4^2
- 21600 x_1 x_3 x_4^2
\\ &\quad
+ 16200 x_4^3
- 288 x_1^7 x_5
- 2592 x_1^5 x_2 x_5
+ 4320 x_1^3 x_2^2 x_5
- 12960 x_1 x_2^3 x_5
+ 5760 x_1^4 x_3 x_5
\\ &\quad
- 17280 x_2^2 x_3 x_5
+ 23040 x_1 x_3^2 x_5
+ 8640 x_1^3 x_4 x_5
+ 25920 x_1 x_2 x_4 x_5
- 34560 x_3 x_4 x_5
\\ &\quad
- 20736 x_1^2 x_5^2
+ 20736 x_2 x_5^2
+ 1200 x_1^6 x_6
- 3600 x_1^4 x_2 x_6
+ 10800 x_1^2 x_2^2 x_6
+ 10800 x_2^3 x_6
\\ &\quad
- 9600 x_1^3 x_3 x_6
- 28800 x_1 x_2 x_3 x_6
+ 19200 x_3^2 x_6
+ 21600 x_1^2 x_4 x_6
- 21600 x_2 x_4 x_6.
\end{split}}\normalsize
\end{equation*}
Finally, at $n=16$ we select $D=32$ and compute a polynomial
$F_{16}(x_1,\ldots,x_7)$ with $77$ terms and height $14515200$.
We remark that the set $\mathcal{M}'_{32,16}$ has $164$ elements, so our
relation relies on less than half of these monomials.

With the polynomials $F_n$ in hand, our strategy for even $n$ now mirrors
that for odd $n$.
We use the restrictions in \eqref{eqnEvenabSelection} to select $a_1,
\ldots, a_s$ and $b_1, \ldots, b_t$ for a total of $n/2+2$ selections.
After the penultimate selection we perform local checks modulo two or three
primes $q_i$, which are also required divisors of $C_n'$, and we ensure our
last selection survives both the check mod $r$ as before, with $r$ again an
unused prime, as well as the local conditions modulo the $q_i$ again.
We then test if $F_n(m'_2,\ldots,m'_{n-2})=0$ mod $M_1=2^{50}$ as well as
$M_2=2^{30}+3$, for the appropriate signed moment sums $m_2', \ldots,
m_{n-2}'$ given by the right side of \eqref{eqnEvenMoments}.
Surviving sets are then tested over $\mathbb{Z}$ using Magma, and then a
Gr\"obner basis is calculated to determine if the polynomial system
\eqref{eqnEvenMoments} has a solution in integers (or even in algebraic
irrationals of degree at most $2\floor{(n-2)/4}$).

\subsubsection*{Results}

Our results for even $n$ are also summarized in Table~\ref{tableZComp}.
We add some additional details here.

\subsubsection*{$n=10$:}
We first checked that no solutions over
$\mathbb{Z}$ exist with height at most $2500$ that reduce mod $29$ to
one of the symmetric ideal solutions for this prime.  There are only
two such solutions, up to unit multiples: the Chebyshev example from
Theorem~\ref{thmEvenLocal}(\ref{enumEvenA}) with $(c_1,c_2)=(1,14)$,
so $\pm\{1,3,3,12,12\}$, $\pm\{2,4,7,10,14\}$, and the sporadic
solution $\pm\{0,1,1,5,5\}$, $\pm\{2,6,9,10,11\}$.  We then used our
method with $p_1=29$ and $p_2=23$ to verify that no additional
primitive symmetric ideal solutions exist with height at most $2500$,
beyond the two solutions \eqref{eqnPTE10} found in \cite{BLP}.

Our search recovers several thousand pairs $(A,B)$ consisting of at least
$14$ integers, plus a few quadratic or quartic irrationals.
A small number of these involve only quadratic irrationals, including some
where all the non-integral values have the form $\sqrt{k}$ for some $k$,
and others that lie in a single quadratic extension of the rationals.
Many more such examples surely exist; the ones we found were simply those
that survived all of our local requirements for integral solutions.
We list two examples of each type below:
\begin{gather*}
\pm\{13, 77, 83, 107, i\sqrt{5099}\},\;
\pm\{47, 49, 97, 103, i\sqrt{5291}\};\\
\pm\{823, 905, \sqrt{2647921}, 2079, 2417\},\;
\pm\{\sqrt{328849}, 1265, 1431, 2137, 2401\};\\
\pm\{20, 37, 72, 76, 91\},\;
\pm\{9, 47, 62, \abs{\beta_1}, \abs{\beta_2}\};\\
\pm\{389, 889, 1742, 2112, 2334\},\;
\pm\{213, 1001, \abs{\beta_3}, \abs{\beta_4}, 2308\}.
\end{gather*}
Here $\beta_1$ and $\beta_2$ are the roots of $x^2+6x-7420$, and $\beta_3$
and $\beta_4$ are the roots of $x^2+500x-3630036$.

\subsubsection*{$n=12$:}
We search to height $1511$ by using $(p_1,p_2)=(31,29)$.
Two local solutions mod $31$ obstruct our use of that prime:
$\pm\{0,1,1,1,8,13\}$, $\pm\{3,7,7,12,15,15\}$, and $\pm\{0,0,0,1,5,6\}$,
$\pm\{3,3,13,13,15,15\}$.
We first verified that no integer solutions with height at most $1511$
exist which reduce mod $31$ to a unit multiple of one of these two
multisets.
After this, we completed a search over $\mathbb{Z}$ to this height using
$(p_1,p_2)=(31,29)$.
We find that \eqref{eqnPTE12a} and \eqref{eqnPTE12b} are the only two
primitive symmetric ideal solutions to the \PTE\ problem at this size with
height at most $1511$.

We remark that the method of Broadhurst \cite{Broadhurst} that produced
\eqref{eqnPTE12c} determined the symmetric ideal solutions with size $n=12$
and height at most $41\cdot53=2173$, for which the associated constant
$C_{12}(A,B)$ is relatively prime to $2173$.
In \eqref{eqnPTE12b}, both $41$ and $53$ divide this constant.

Our searches also find a few examples where $A$ has ten integers plus
the roots of a quadratic and $B$ has eight integers and the roots of a
quartic, such as
\begin{equation*}
\begin{split}
A &= \pm\{73, 279, \sqrt{137985}, 547, 661, 715\},\\
B &= \pm\{155, 197, 671, 711\} \cup \{\alpha : f(\alpha) = 0\},
\end{split}
\end{equation*}
where $f(x) = x^4 - 449914 x^2 + 47934959865$.
A few thousand examples also arise where both $A$ and $B$ have eight
integers, along with four quadratic or quartic irrationals.

\subsubsection*{$n=14$:}
We find that no solutions exist with height at most $900$ using
$(p_1,p_2)=(37,31)$ from Table~\ref{tableReqDivZ}.
Only two solutions that involve irrational values arise: in both of them
one set has ten integers and four quartic irrationals, and the other has
eight integers and the roots of a sextic polynomial.
One of these is
\begin{equation*}
\begin{split}
A &= \pm\{126, 197, 260, 331, 768\} \cup \{\alpha : f_A(\alpha) = 0\},\\
B &= \pm\{12, 83, 334, 405\} \cup \{\alpha : f_B(\alpha) = 0\},
\end{split}
\end{equation*}
where $f_A(x) = x^4 - 160333x^2 - 648436092$ and $f_B(x) = x^6 - 699389x^4
+ 67656970756 x^2 - 1788905080155648.$

\subsubsection*{$n=16$:}
We used $(p_1,p_2)=(53,43)$ to show that no integer solutions of
this size exist with height at most $850$.
In addition, none of the candidates that survive the local tests satisfies
the polynomial $F_{16}$, so we find no algebraic solutions having at least
ten integer values.

\section{Divisibility Requirements in Some Quadratic Number Fields}\label{secDivisQ}

Caley~\cite{Caley12,Caley13} proved that many of the statements regarding
required divisors in the \PTE\ problem generalize in a natural way to the
setting of the ring of integers in a number field.
Of course, one must work in a unique factorization domain in order for
these quantities to be well defined.
While Caley concentrated on \PTE\ over the Gaussian integers, he recorded a
number of statements that are valid in any \UFD\@.
We summarize these here, and also note some additional requirements that
arise in the symmetric case.

Following our prior development over $\mathbb{Z}$, if $(A,B)$ forms an
ideal solution to the \PTE\ problem of size $n$ over a \UFD\ $\mathcal{O}$,
we let $C_n(\mathcal{O},A,B)$ denote the constant as in \eqref{eqnCnZ}.
Then we define $C_n(\mathcal{O})$ as the greatest common divisor in
$\mathcal{O}$ over all such values $C_n(\mathcal{O},A,B)$.
Caley established the following generalization of Proposition~\ref{propRS}.

\begin{prop}[Caley \cite{Caley13}]\label{propCaley1}
Suppose $\mathcal{O}$ is a \UFD, and let $C_n(\mathcal{O})$ denote the
constant in an ideal \PTE\ solution of size $n$ over the ring
$\mathcal{O}$.
Let $q\in\mathcal{O}$ denote a prime.
\begin{enumerate}[(i)]
\item If $q\mid C_n(\mathcal{O})$ then $q^{\ceil{\frac{n}{N(q)}}}\mid
C_n(\mathcal{O})$.\label{pc1p1}
\item If $N(q) > 3$ then $N(q)\mid C_{N(q)}(\mathcal{O})$.\label{pc1p2}
\item If $n+2\leq N(q) < n+2+\frac{n-3}{6}$ then $q\mid
C_n(\mathcal{O})$.\label{pc1p3}
\item If $N(q)=2$ then $q^4\mid C_5(\mathcal{O})$.\label{pc1p4}
\end{enumerate}
\end{prop}

Caley also obtained an analogue of Kleiman's result \cite{Kleiman}
for quadratic number fields: in general, primes that split or ramify are
included, though sometimes with a smaller multiplicity.

\begin{prop}[Caley \cite{Caley13}]\label{propCaley2}
Suppose $p<n$ is a rational prime, and $\mathcal{O}$ is the ring of
integers in a number field of degree $2$.
Let $s=\floor{n/p}$, suppose $p^\ell \;\|\; n$, and suppose $s\geq\ell$.
\begin{enumerate}[(i)]
\item If $p$ splits in $\mathcal{O}$ as $p=\pi_1\pi_2$ then
$p^{s-\ell} \mid C_n(\mathcal{O})$.
\item If $p$ ramifies in $\mathcal{O}$ as $p=\pi^2$ then $\pi^{s-\ell} \mid
C_n(\mathcal{O})$.
\end{enumerate}
\end{prop}

Using these facts, we can enumerate the required divisors for the \PTE\
constant over any quadratic number field whose ring of integers is a \UFD\@.
In this article we treat imaginary quadratic fields with class number $1$,
so using these propositions, together with the second column of
Table~\ref{tableReqDivZ}, we can record required divisors for
$C_n(\mathcal{O})$ in each ring we consider.
These appear in Tables~\ref{tableReqDivQ1} and~\ref{tableReqDivQ2} for $5\leq
n\leq 20$ for the number fields $\mathbb{Q}(i)$, $\mathbb{Q}(i\sqrt{2})$,
$\mathbb{Q}(\omega)$, and $\mathbb{Q}(i\sqrt{7})$.
Here and throughout, $\omega$ denotes the primitive third root of unity
$e^{2\pi i/3}$.
In essence, every rational prime from the original table for
$C_n(\mathbb{Z})$ that splits or ramifies in $\mathcal{O}$ appears in the
table for $\mathcal{O}$.
We use $\rho_p$ in Tables~\ref{tableReqDivQ1} and~\ref{tableReqDivQ2} for a
required factor of a rational prime $p$ that ramifies in $\mathcal{O}$.
Inert primes occasionally appear as well, when they satisfy part
(\ref{pc1p3}) of Proposition~\ref{propCaley1}.

We define $C_n'(\mathcal{O})$ in the same way for symmetric ideal \PTE\ 
solutions over $\mathcal{O}$ of size $n$.
In this case, we also inherit the primes from Table~\ref{tableReqDivZ} that
split in $\mathcal{O}$: suppose $p$ splits in $\mathcal{O}$ as
$p=\pi_1\pi_2$, and that $p$ appears in the last column of
Table~\ref{tableReqDivZ} as a required divisor for $n$ in the symmetric
case.
Suppose also that $(A,B)$ is a symmetric solution of size $n$ over $\mathcal{O}$
where $\pi_1\nmid C_n(\mathcal{O},A,B)$.
Then we can reduce this equation mod $\pi_1$ to obtain a solution with
nonzero constant over $\mathcal{O}/(\pi_1)\cong\mathbb{F}_p$, so we can map
this to a solution over $\mathbb{F}_p$ with nonzero constant.
By Lemma~\ref{lemRSML} there exists a polynomial $h\in\mathbb{F}_p[x]$ with the
properties stated there, but we already verified no such polynomial can
exist when we showed that $p\mid C_n'(\mathbb{Z})$.
Therefore $\pi_1$ (and $\pi_2$, hence $p$) must divide $C_n'(\mathcal{O})$.

Required divisors in $C_n'(\mathcal{O})$ also appear in
Tables~\ref{tableReqDivQ1} and~\ref{tableReqDivQ2} in boldface.
We remark that Table~\ref{tableReqDivQ1} includes a number of required
divisors for $\mathbb{Z}[i]$ that are missing from \cite[Table~1]{Caley13},
which reported on values for $n\leq15$.
Part of this is the fact that Caley's table does not include the primes
required for the symmetric case.
However, some required factors are missing for the broader ideal case,
including $(3+2i)(3-2i)=13$ at $n=8$ and $n=11$, and $(4+i)(4-i)=17$ at
$n=10$, $11$, $12$, and $15$.
Also, we record larger multiplicities (using
Proposition~\ref{propCaley1}(\ref{pc1p1})) at a number of primes in our table,
compared to those listed in \cite{Caley13}.
New information determined here on required divisors for $n\leq15$ in the case
of the Gaussian integers is displayed in a boxed format in
Table~\ref{tableReqDivQ1}.

\begin{table}[tbp]
\caption{Required divisors in $C_n(\mathcal{O})$ and
$C_n'(\mathcal{O})$, for $\mathcal{O}=\mathbb{Z}[i]$ and
$\mathbb{Z}[i\sqrt{2}]$.
Boldface factors denote required divisors for $C_n'(\mathcal{O})$ only.
Boxed entries for $n\leq15$ in the table for the Gaussian integers indicate
factors reported here for the first time (not recorded in
\cite{Caley13}).}\label{tableReqDivQ1}
\begin{center}
\small
\begin{tabular}{|c|c|c|}\hline
\TS$n$ & $\mathbb{Z}[i]$ ($\rho_2=1+i$) & $\mathbb{Z}[i\sqrt{2}]$ ($\rho_2=i\sqrt{2}$)\\\hline
\TS 5 & $\rho_2^4\cdot5$ &
      $\rho_2^4\cdot3^2$\\
 6  & $\rho_2^3\cdot\boxed{5^2}$ &
      $\rho_2^3\cdot3^2$\\
 7  & $\rho_2^4\cdot3\cdot\boxed{5^2}$ &
      $\rho_2^3\cdot3^3\cdot11\cdot\mathbf{19}$\\
 8  & $\rho_2^4\cdot\boxed{5^2\cdot13}$ &
      $3^3\cdot11$\\
 9  & $\rho_2^5\cdot3^2\cdot\boxed{5^2}\cdot13$ &
      $\rho_2^5\cdot3^3\cdot11$\\
10  & $\rho_2^5\cdot\;\boxed{5^2}\cdot13\cdot\boxed{17}$ &
      $\rho_2^5\cdot3^4\cdot17$\\
11  & $\rho_2^6\cdot\boxed{5^3\cdot13\cdot17}$ &
      $\rho_2^6\cdot3^4\cdot11\cdot17$\\
12  & $\rho_2^6\cdot\boxed{5^3\cdot17\cdot\mathbf{29}}$ &
      $\rho_2^6\cdot3^4\cdot11^2\cdot17\cdot19$\\
13  & $\rho_2^7\cdot\boxed{5^3}\cdot13\cdot17\cdot\boxed{\mathbf{29\cdot37\cdot 41}}$ &
      $\rho_2^7\cdot3^5\cdot11^2\cdot17\cdot\mathbf{41}$\\
14  & $\rho_2^7\cdot\boxed{5^3\cdot13^2}\cdot17\cdot\boxed{\mathbf{37}}$ &
      $\rho_2^7\cdot3^5\cdot11^2\cdot17\cdot19$\\
15  & $\rho_2^8\cdot\boxed{5^3\cdot13^2\cdot17\cdot\mathbf{37\cdot41}}$ &
      $\rho_2^8\cdot3^5\cdot11^2\cdot19\cdot\mathbf{41\cdot43}$\\
16  & $\rho_2^8\cdot5^4\cdot13^2\cdot\mathbf{29\cdot37\cdot41\cdot53}$ &
      $3^6\cdot11^2\cdot19\cdot\mathbf{41\cdot43}$\\
17  & $\rho_2^9\cdot5^4\cdot13^2\cdot17\cdot29\cdot\mathbf{37\cdot41\cdot53}$ &
      $\rho_2^9\cdot3^6\cdot11^2\cdot17\cdot19\cdot\mathbf{41\cdot43}$\\
18  & $\rho_2^9\cdot5^4\cdot13^2\cdot17^2\cdot29\cdot\mathbf{41}$ &
      $\rho_2^9\cdot3^6\cdot11^2\cdot17^2\cdot\mathbf{41\cdot43\cdot59}$\\
19  & $\rho_2^{10}\cdot5^4\cdot13^2\cdot17^2\cdot29\cdot\mathbf{41\cdot53}$ &
      $\rho_2^{10}\cdot3^7\cdot11^2\cdot17^2\cdot19\cdot\mathbf{41\cdot43\cdot59}$\\
20  & $\rho_2^{10}\cdot5^4\cdot13^2\cdot17^2\cdot29\cdot\mathbf{37\cdot53}$ &
      $\rho_2^{10}\cdot3^7\cdot11^2\cdot17^2\cdot19^2\cdot\mathbf{43}$\\\hline
\end{tabular}
\end{center}
\end{table}

\begin{table}[tbp]
\caption{Required divisors in $C_n(\mathcal{O})$ and
$C_n'(\mathcal{O})$, for $\mathcal{O}=\mathbb{Z}[e^{2\pi i/3}]$ and
$\mathbb{Z}[(1+i\sqrt{7})/2]$.
Boldface factors denote required divisors for $C_n'(\mathcal{O})$ 
only.}\label{tableReqDivQ2} 
\begin{center}
\small
\begin{tabular}{|c|c|c|}\hline
\TS$n$ & $\mathbb{Z}[\omega]$ ($\rho_3=i\sqrt{3}$) &
$\mathbb{Z}[(1+i\sqrt{7})/2]$ ($\rho_7=i\sqrt{7}$)\\\hline
\TS 5 & $\rho_3^2\cdot7$ &
   $2^4$\\
6  & &
   $2^3$\\
7  & $\rho_3^3\cdot7\cdot\mathbf{19}$ &
   $2^4\cdot3\cdot\rho_7^2\cdot11$\\
8  & $\rho_3^3\cdot7^2\cdot13$ &
   $11$\\
9  & $7^2\cdot13$ &
   $2^5\cdot3^2\cdot11$\\
10 & $\rho_3^4\cdot7^2\cdot13$ &
   $2^5\cdot23$\\
11 & $\rho_3^4\cdot7^2\cdot13\cdot19\cdot\mathbf{31}$ &
   $2^6\cdot11$\\
12 & $\rho_3^4\cdot7^2\cdot19$ &
   $2^6\cdot11^2\cdot\mathbf{29}$\\
13 & $\rho_3^5\cdot7^2\cdot13\cdot19\cdot\mathbf{31\cdot37}$ &
   $2^7\cdot11^2\cdot23\cdot\mathbf{29\cdot37}$\\
14 & $\rho_3^5\cdot7^2\cdot13^2\cdot19\cdot\mathbf{31\cdot37}$ &
   $2^7\cdot11^2\cdot23\cdot\mathbf{37}$\\
15 & $\rho_3^5\cdot7^3\cdot13^2\cdot19\cdot\mathbf{37\cdot43}$ &
   $2^8\cdot11^2\cdot23\cdot\mathbf{37\cdot43}$\\
16 & $\rho_3^6\cdot7^3\cdot13^2\cdot19\cdot\mathbf{37\cdot43}$ &
   $11^2\cdot23\cdot\mathbf{29\cdot37\cdot43\cdot53}$\\
17 & $\rho_3^6\cdot7^3\cdot13^2\cdot19\cdot\mathbf{31\cdot37\cdot43}$ &
   $2^9\cdot11^2\cdot29\cdot\mathbf{37\cdot43\cdot53}$\\
18 & $\rho_3^7\cdot7^3\cdot13^2\cdot\mathbf{31\cdot43}$ &
   $2^9\cdot11^2\cdot29\cdot\mathbf{43}$\\
19 & $\rho_3^7\cdot7^3\cdot13^2\cdot19\cdot\mathbf{43}$ &
   $2^{10}\cdot11^2\cdot23\cdot29\cdot\mathbf{43\cdot53}$\\
20 & $\rho_3^7\cdot7^3\cdot13^2\cdot19^2\cdot31\cdot\mathbf{37\cdot43}$ &
   $2^{10}\cdot11^2\cdot23\cdot29\cdot\mathbf{37\cdot43\cdot53}$\\\hline
\end{tabular}
\end{center}
\end{table}

\section{Computations in Some Imaginary Quadratic Number Fields}\label{secCompQ}

We describe some searches for symmetric ideal solutions in the \PTE\
problem over a number of quadratic number fields where the ring of integers
is a \UFD\@.
We further restrict our attention to qualifying imaginary quadratic number
fields, since the unit groups here are finite.
Thus, we need to use one of $\mathbb{Q}(i\sqrt{d})$ with $d\in\{1, 2, 3, 7,
11, 19, 43, 163\}$.
We concentrate here on the first six such fields.

Our method is similar to the sieving procedure described over $\mathbb{Z}$
in \cite{BLP} and employed in \cite{Caley13} for $\mathbb{Z}[i]$.
We require the following notation.
Let $\mathcal{O}$ denote the ring of integers of the field being
considered, and let $U$ denote its group of units.
For a fixed $n$, let $p$ denote the largest rational prime that is known to
be required to divide the constant $C'_n(\mathcal{O})$: this is listed in
Table~\ref{tableReqDivQ1} or~\ref{tableReqDivQ2} for $d\in\{1,2,3,7\}$.
In the cases under consideration, this prime splits in $\mathcal{O}$, and
we write $\pi_1$ and $\pi_2$ for its two prime factors in this ring.
$H$ is a positive integer governing the size of the search space.
We first describe the algorithm for odd~$n$.

\begin{alg}\label{algPTEQ1}
Searching for symmetric ideal solutions of odd size $n$ in the \PTE\
problem over the \UFD\ $\mathcal{O}$.
\end{alg}

\begin{iolist}{Output}
\item[Input]
An odd integer $n$, primes $\pi_1$, $\pi_2$ in $\mathcal{O}$ as above, and
a positive integer $H$.
\item[Output]
Symmetric ideal solutions of size $n$ in $\mathcal{O}$, with elements
selected subject to the parameter $H$.
\item[Description]
\end{iolist}

\begin{biglabellist}{Step~5}
\item[Step~1]
Construct $S_1$, the set of all nonzero multiples of $\pi_1$ in
$\mathcal{O}$ with norm at most $H$, keeping just one representative with
respect to unit multiples, and create $T_1=\{su : s\in S_1,\, u\in U\}$.
Construct $S_2$ and $T_2$ in the same way using the prime $\pi_2$.
Finally, create $V$, the set of nonzero algebraic integers from
$\mathcal{O}$ having norm at most $H/p$ (with $p=\pi_1\pi_2$).

\item[Step~2]
For each possible selection of $a_1\in S_1$, perform Step~3.

\item[Step~3]
Assume $a_1, \ldots, a_k$ have been selected.
If $k=(n+1)/2$, perform Step~4.
Otherwise, let $j=1$ if $k$ is even and $j=2$ if $k$ is odd.
If $\pi_j \mid \sum_{i=0}^{k-1} a_i$, then for each possible choice of
$a_{k+1}\in V$, perform Step~3.
Otherwise, for each possible choice of $t\in T_j$, set $a_{k+1}=t-a_k$
(disallowing $0$ or $\pm a_i$ for $i\leq k$) and perform Step~3.

\item[Step~4]
For each prime $\pi\in\mathcal{O}$ with norm less than $p$ which is a
required divisor for $C_n'(\mathcal{O})$, test if either $\pi\mid a_j$ for
some $j$, or $\pi \mid (a_i+a_j)$ for some pair $i$, $j$ with $1\leq
i<j\leq(n+1)/2$.
This is required if the current elements can be extended to a symmetric
ideal solution.
If this fails for some prime $\pi$, reject this choice and continue the
enumeration.
If it succeeds for each such prime, perform Step~5.

\item[Step~5]
Using Lagrange interpolation, construct the polynomial $q_1(x)$ as in
\eqref{eqnPolyInterp1}, and test if this polynomial splits over
$\mathcal{O}$.
If it does, report its roots as a symmetric ideal \PTE\ solution of size
$n$ in this ring, otherwise reject this set and continue the enumeration.
\end{biglabellist}

The algorithm for even $n$ is similar.
Here, we may assume $a_i\equiv b_i \bmod \pi$ for each prime $\pi\mid
C_n'(\mathcal{O})$, so we may demand that
\[
a_i^2\equiv b_i^2 \pmod{\pi_1}
\]
and
\[
(a_{i+1}^2-b_i^2)\sum_{j=1}^i (a_j^2-b_j^2) \equiv 0 \pmod{\pi_2}
\]
in making our selections.

\begin{alg}\label{algPTEQ0}
Searching for symmetric ideal solutions of even size $n$ in the \PTE\
problem over $\mathcal{O}$.
\end{alg}

\begin{iolist}{Output}
\item[Input]
An even integer $n$, primes $\pi_1$, $\pi_2$ in $\mathcal{O}$ as above, and
positive integers $H$ and $\ell$.
\item[Output]
Symmetric ideal solutions of size $n$ in $\mathcal{O}$, with elements
selected subject to the parameters $H$ and $\ell$.
\item[Description]
\end{iolist}

\begin{biglabellist}{Step~5}
\item[Step~1]
In addition to $S_1$, $T_1$, and $V$ as in Algorithm~\ref{algPTEQ1},
construct the set $S_0$ consisting of the nonzero algebraic integers in
$\mathcal{O}$ with norm at most $H/\ell$, keeping just one representative
with respect to unit multiples.

\item[Step~2]
For each possible selection of $a_1\in S_0$, perform Step~3.

\item[Step~3]
Assume $k$ elements have been selected: $a_1, \ldots, a_{\ceil{k/2}}$
and $b_1, \ldots,\allowbreak b_{\floor{k/2}}$.
If $k=(n+1)/2$, perform Step~4.
Otherwise, suppose first $k$ is odd, so we must select $b_{(k+1)/2}$.
For each $t\in T_1$ set $b_{(k+1)/2}=a_{(k+1)/2}+t$ and invoke Step~3.
On the other hand, if $k$ is even, so that we must next select $a_{k/2+1}$,
first test if $\pi_2 \mid \sum_{i=1}^{k/2} (a_i^2-b_i^2)$.
If this holds, then for each $v\in V$ set $a_{k/2+1}=v$ and invoke Step~3.
If not, then for each $t\in T_2$ set $a_{k/2+1}=b_{k/2}+t$ and call Step~3.
Throughout, we take care to avoid selecting the value $0$, or a value
already selected, or the negation of a value already selected.

\item[Step~4]
For each prime $\pi\in\mathcal{O}$ with norm less than $p$ which is a
required divisor for $C_n'(\mathcal{O})$, test if $\pi\mid(a_i^2-b_j^2)$
for some $i$ and $j$.
If this fails for some prime $\pi$, reject this choice and continue the
enumeration.
If it succeeds for each such prime, perform Step~5.

\item[Step~5]
Using Lagrange interpolation, construct the polynomial $q_0(x)$ as in
\eqref{eqnPolyInterp0} with $k=\floor{n/4}$, and test if this polynomial
splits over $\mathcal{O}$.
If it does, we perform this process a second time, reversing the roles of
the $a_i$ and $b_j$ in \eqref{eqnPolyInterp0} and setting $k=\ceil{n/4}$,
and check if this polynomial splits as well.
If both succeed, then we can recover the remaining $a_i$ values from the
first polynomial, and the remaining $b_j$ values from the second, and we
have a symmetric ideal \PTE\ solution of size $n$ in this ring.
Otherwise, we reject this set and continue the enumeration.
\end{biglabellist}

\begin{table}[tbp]
\caption{Searches performed for symmetric ideal \PTE\ solutions
in $\mathbb{Q}(i\sqrt{d})$.}\label{tableQNFSearches}
\small
\begin{tabular}{c@{\quad}c@{\quad}c}
\begin{tabular}[t]{|c|cccc|}\hline
\TS $d$ &  $n$ &  $p$ &    $H$ & $\ell$\\\hline
     &  $9$ & $13$ & $1400$ & -- \\
     & $10$ & $17$ &  $500$ & $4$ \\
     & $11$ & $17$ &  $750$ & -- \\
 $1$ & $12$ & $29$ &  $500$ & $7$ \\
     & $13$ & $41$ &  $680$ & -- \\
     & $14$ & $37$ & $300$ & $8$\\
     & $16$ & $53$ & $250$ & $12$\\\hline
     &  $9$ & $11$ & $1250$ & -- \\
     & $10$ & $17$ &  $800$ & $8$ \\
 $2$ & $11$ & $17$ & $1000$ & -- \\
     & $12$ & $19$ &  $450$ & $9$ \\
     & $13$ & $41$ & $1000$ & -- \\\hline
\end{tabular} &
\begin{tabular}[t]{|c|cccc|}\hline
\TS $d$ &  $n$ &  $p$ &    $H$ & $\ell$\\\hline
     &  $9$ & $13$ & $1300$ & -- \\
     & $10$ & $13$ &  $400$ & $2.2$ \\
     & $11$ & $31$ & $1250$ & -- \\
 $3$ & $12$ & $19$ &  $250$ & $3$ \\
     & $13$ & $37$ &  $500$ & -- \\
     & $14$ & $37$ & $250$ & $5$\\
     & $15$ & $43$ & $350$ & -- \\\hline
     &  $9$ & $11$ & $1200$ & -- \\
     & $10$ & $23$ & $1000$ & $12$ \\
 $7$ & $11$ & $11$ &  $550$ & -- \\
     & $12$ & $29$ &  $600$ & $12$ \\
     & $13$ & $37$ &  $800$ & -- \\\hline
\end{tabular} &
\begin{tabular}[t]{|c|cccc|}\hline
\TS $d$ &  $n$ &  $p$ &    $H$ & $\ell$\\\hline
     &  $9$ & $11$ & $1250$ & -- \\
     & $10$ & $23$ & $1100$ & $11$ \\
$11$ & $11$ & $31$ & $1800$ & -- \\
     & $12$ & $11$ &  $160$ & $5$ \\
     & $13$ & $37$ &  $900$ & -- \\\hline
     &  $9$ & $11$ & $2100$ & -- \\
     & $10$ & $23$ & $1500$ & $11$ \\
$19$ & $11$ & $19$ & $1400$ & -- \\
     & $12$ & $19$ &  $500$ & $9$ \\
     & $13$ & $23$ &  $800$ & -- \\\hline
\end{tabular}
\end{tabular}
\end{table}

Table~\ref{tableQNFSearches} records the searches we
performed using SageMath~\cite{Sage} for symmetric ideal solutions in
the \PTE\ problem over $\mathbb{Q}(i\sqrt{d})$ for
$d\in\{1,2,3,7,11,19\}$ and $n\in\{9,10,11,12,13\}$, as well as a few
larger values of $n$ for $d=1$ and $d=3$.  These searches uncovered a
number of primitive solutions, i.e., solutions having no common factor
in the corresponding ring of integers $\mathcal{O}$.  Over the
Gaussian integers, while we find no solutions at the odd values of $n$
tested, we find five solutions at $n=10$:
\begin{gather}
\begin{split}\label{eqnGI10a}
A&=\pm\{3+4i, 4+4i, 4-4i, 6-3i, 6+i\},\;\\
  B&=\pm\{3i, 4+5i, 5-4i, 6-2i, 6+2i\};
\end{split}\\
\begin{split}\label{eqnGI10b}
A&=\pm\{2-11i, 3-8i, 4-9i, 7-5i, 7+i\},\;\\
  B&=\pm\{1-11i, 4-11i, 5, 6-i, 7-7i\};
\end{split}\\\label{eqnGI10c}
A=\pm\{3+4i, 4-5i, 5+4i, 6-3i, 6+i\},\;
  B=\overline{A};\\\label{eqnGI10d}
A=\pm\{3+14i, 4-12i, 10-7i, 10+i, 13+2i\},\;
  B=i\overline{A};\\\label{eqnGI10e}
A=\pm\{1-10i, 8-25i, 17-22i, 22-6i, 24-13i\},\;
  B=i\overline{A}.
\end{gather}
The first two of these were found by Caley \cite{Caley13}, although the
set corresponding to the solution~\eqref{eqnGI10a} reported in \cite{Caley13} is not
primitive (each element is a multiple of $1+i$).
These five solutions are illustrated in Figure~\ref{figGIpics}(a--e).

At $n=12$ in $\mathbb{Z}[i]$, we find a single primitive solution:
\begin{equation}\label{eqnGI12a}
A = \bigcup_{k=0}^3 i^k \{3+10i, 11+6i, 8+10i\},\;
  B = \overline{A}.
\end{equation}
This set is shown in Figure~\ref{figGIpics}(f).

In $\mathbb{Q}(i\sqrt{2})$, we find one primitive solution at $n=9$.
Writing $\alpha$ for $i\sqrt{2}$, this solution is
\begin{equation}\label{eqnm29a}
\begin{split}
A &= \{-3-3\alpha, -2+3\alpha, -1-5\alpha, \alpha, 1-5\alpha, 2+3\alpha,
3-3\alpha, 4\alpha, 5\alpha\},\\
B &= -A,
\end{split}
\end{equation}
and is displayed in Figure~\ref{figIS2pics}.
No solutions were found here for $10\leq n\leq 13$.

Over the Eisenstein integers, we discover one primitive symmetric ideal
solution in the \PTE\ problem with $n=9$:
\begin{equation}\label{eqnEI9a}
\begin{split}
A &= \{-27\omega-26, -24\omega-54, -21\omega-46, -2\omega-24, 2\omega+14,\\
&\qquad 7\omega+11, 17\omega+43, 23\omega+27, 25\omega+55\},\;
  B = -A.
\end{split}
\end{equation}
It is displayed in Figure~\ref{figEIpics}(a).
No solutions were found in this ring at $n=10$, $11$, or $13$, but we found
solutions here at $n=12$.
First, Algorithm~\ref{algPTEQ0} determined the primitive solution
\begin{equation}\label{eqnEI12a}
A = \bigcup_{j=0}^5 \zeta_6^j \{3, 2\omega+9\},\;
  B = \bigcup_{j=0}^5 \zeta_6^j \{2\omega+11, 7\omega+11\},
\end{equation}
where $\zeta_6$ is a primitive sixth root of unity.
This is displayed in Figure~\ref{figEIpics}(b).

The sixfold symmetry evident here suggested an alternative strategy.
Suppose $z_i=r_i+\omega s_i$ for $1\leq i\leq4$ are Eisenstein integers.
It is straightforward to check that the $k$th moments of $\cup_{j=0}^5
\zeta_6^j \{z_1, z_2\}$ and $\cup_{j=0}^5 \zeta_6^j \{z_3, z_4\}$ match for
$1\leq k\leq 5$ and $7\leq k\leq11$.
Requiring that the sixth moments match as well produces a pair of
equations in the $r_i$ and $s_i$ that need to be satisfied to produce a
symmetric ideal \PTE\ solution in the Eisenstein integers at $n=12$.
We may therefore search for solutions by checking these conditions over a
set of values for the $r_i$ and $s_i$.
To reduce the search space, we may assume that $\arg(z_i)\in[0,\pi/3)$ for
each $i$, $r_1\leq r_2$, and $r_3\leq r_4$.
We also assume that $s_1=0$, so that $z_1$ is a positive rational integer,
and we assume that $r_2\leq r_3$.
A search over $r_4\leq300$ implemented in C++ using this strategy found
another symmetric ideal \PTE\ solution in the Eisenstein integers at
$n=12$: $z_1=27$, $z_2=22(2+\omega)$, $z_3=69$, $z_4=40(2+\omega)$.
Removing the common factor $1+2\omega$ produces the primitive solution
\begin{equation}\label{eqnEI12b}
A = \bigcup_{j=0}^5 \zeta_6^j \{9(1+2\omega), 22\omega\},\;
  B = \bigcup_{j=0}^5 \zeta_6^j \{23(1+2\omega), 40\omega\},
\end{equation}
depicted in Figure~\ref{figEIpics}(c).

We also searched for solutions with $n=18$ in the Eisenstein integers using
a variation of this algorithm, where we begin with six such integers
$z_i=r_i+\omega s_i$, and search for solutions of the equations that arise
from requiring that the sixth and 12th moments agree, as the others up
to the 17th are automatically satisfied.
No solutions were found here, with the $r_i$ increasing and $r_6\leq 64$.

Finally, we add that our searches with $d=7$, $d=11$, and $d=19$ did not
find any symmetric ideal \PTE\ solutions over the ranges tested for $9\leq
n\leq13$.

\begin{figure}[tbp]
\begin{tabular}{ccc}
\includegraphics[width=1.7in]{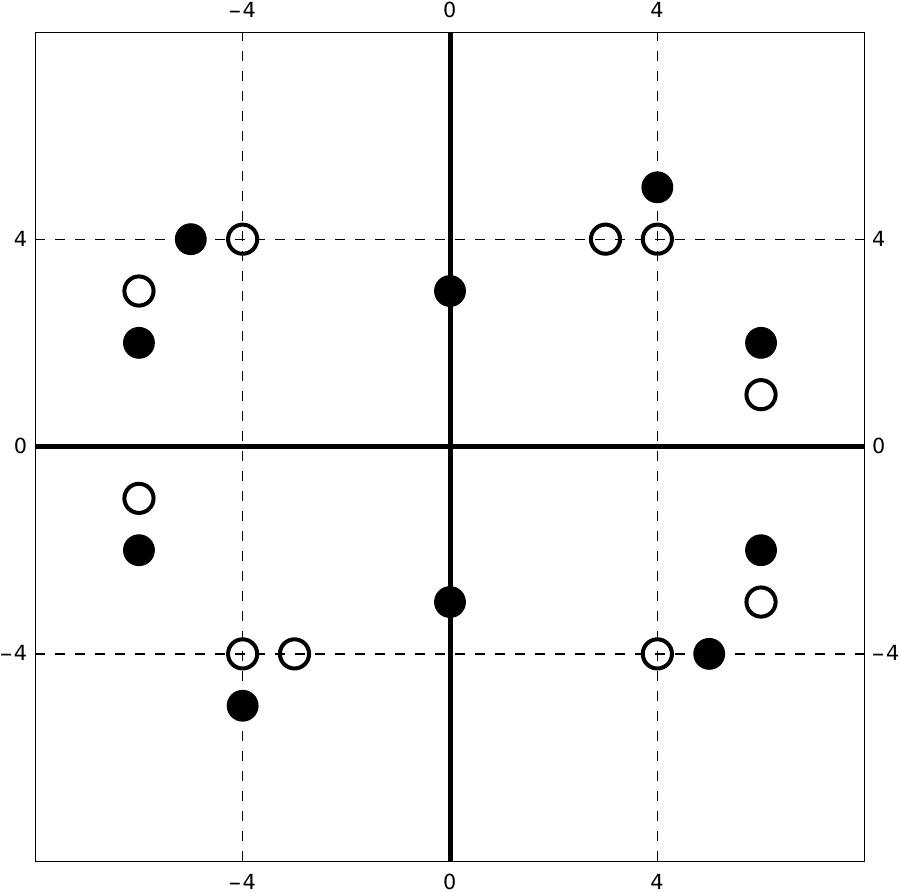} &
\includegraphics[width=1.7in]{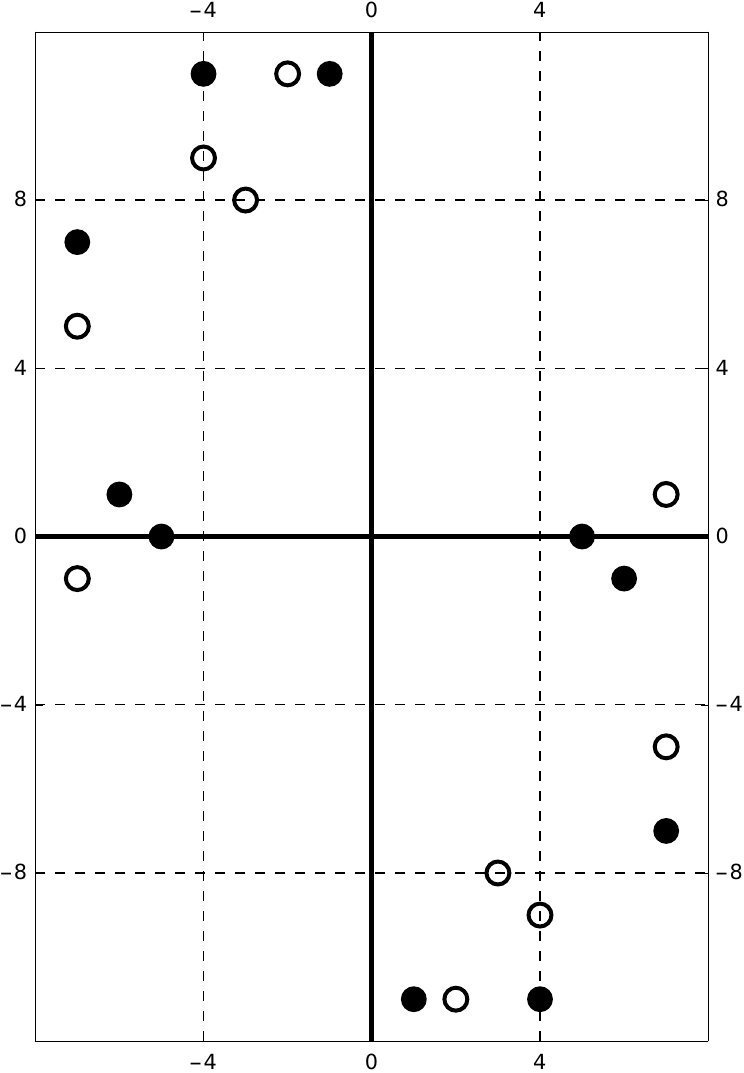} &
\includegraphics[width=1.7in]{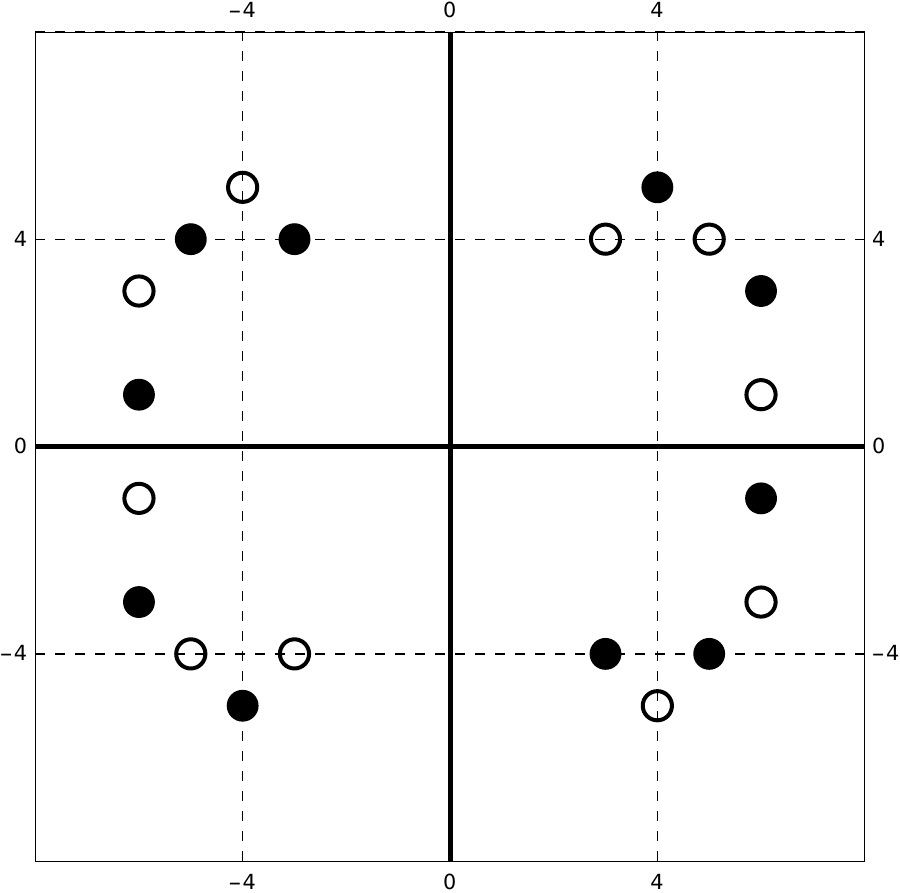}\\
(a) $n=10$ \eqref{eqnGI10a} (Caley) & (b) $n=10$ \eqref{eqnGI10b} (Caley) &
(c) $n=10$ \eqref{eqnGI10c}\\[8pt]
\includegraphics[width=1.7in]{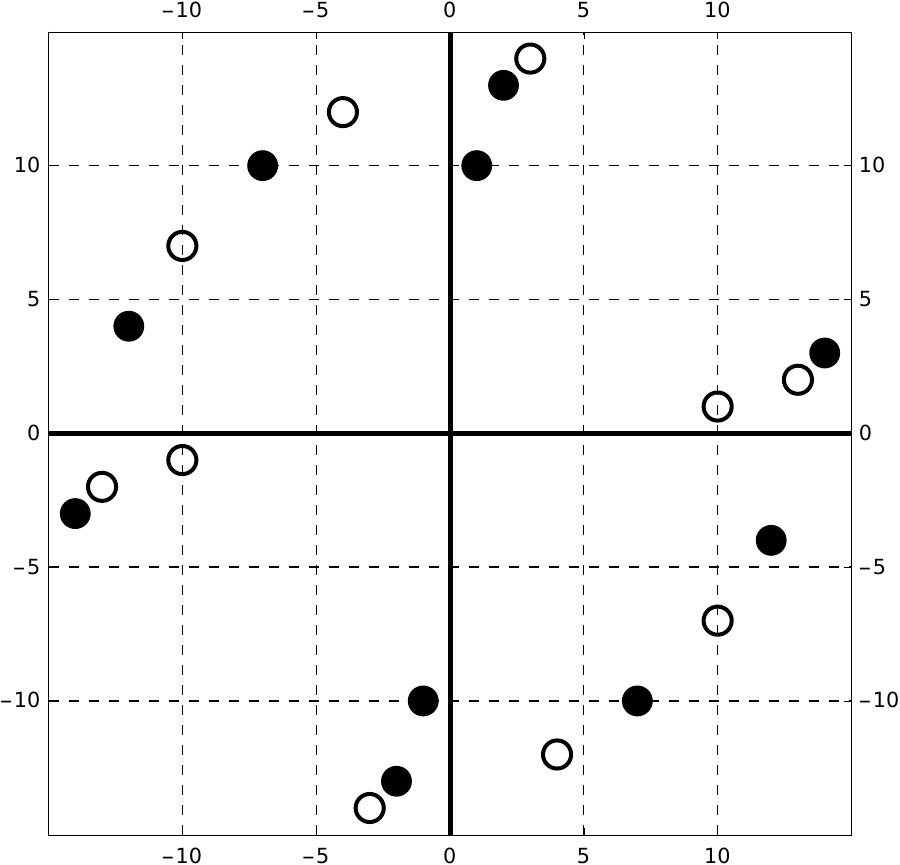} &
\includegraphics[width=1.7in]{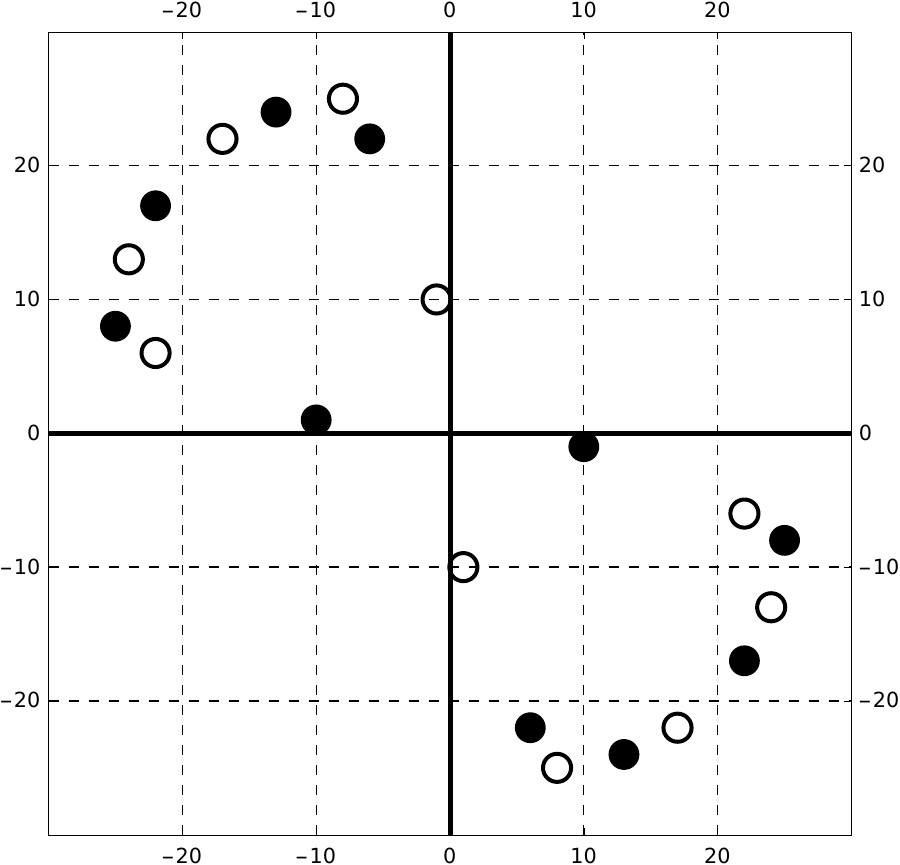} &
\includegraphics[width=1.7in]{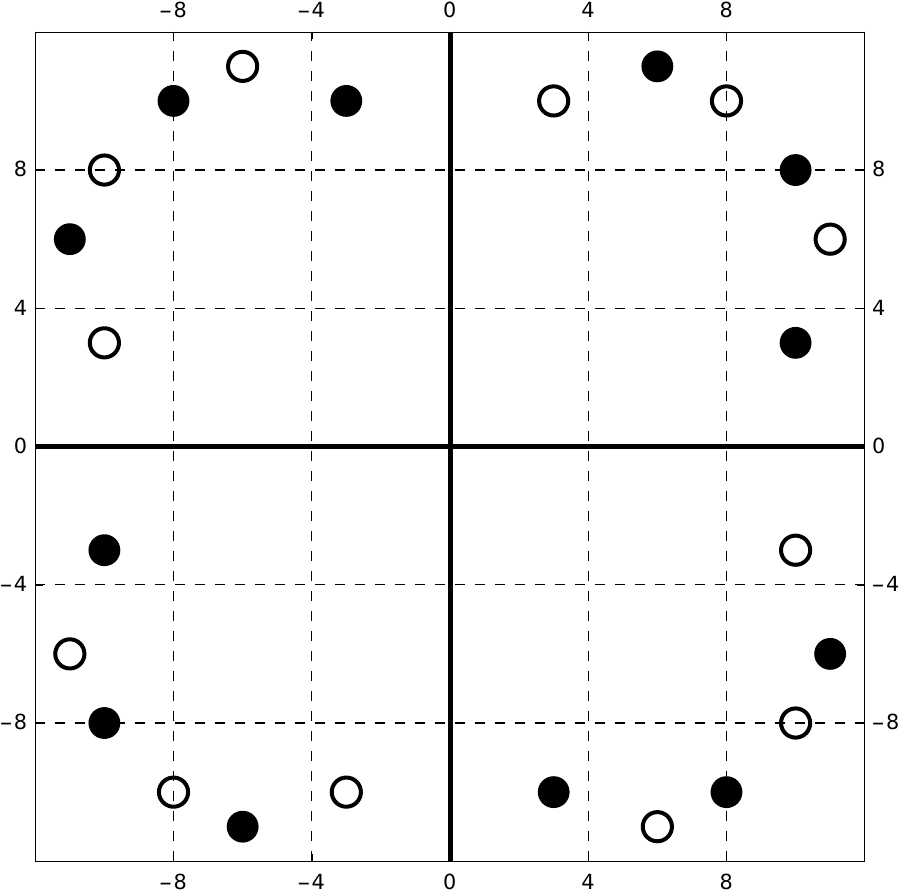}\\
(d) $n=10$ \eqref{eqnGI10d} & (e) $n=10$ \eqref{eqnGI10e} &
(f) $n=12$ \eqref{eqnGI12a}
\end{tabular}
\caption{Symmetric ideal \PTE\ solutions in $\mathbb{Z}[i]$ with
size $n=10$ or $n=12$.}\label{figGIpics}
\end{figure}

\begin{figure}[tbp]
\centering{
\includegraphics[width=1.25in]{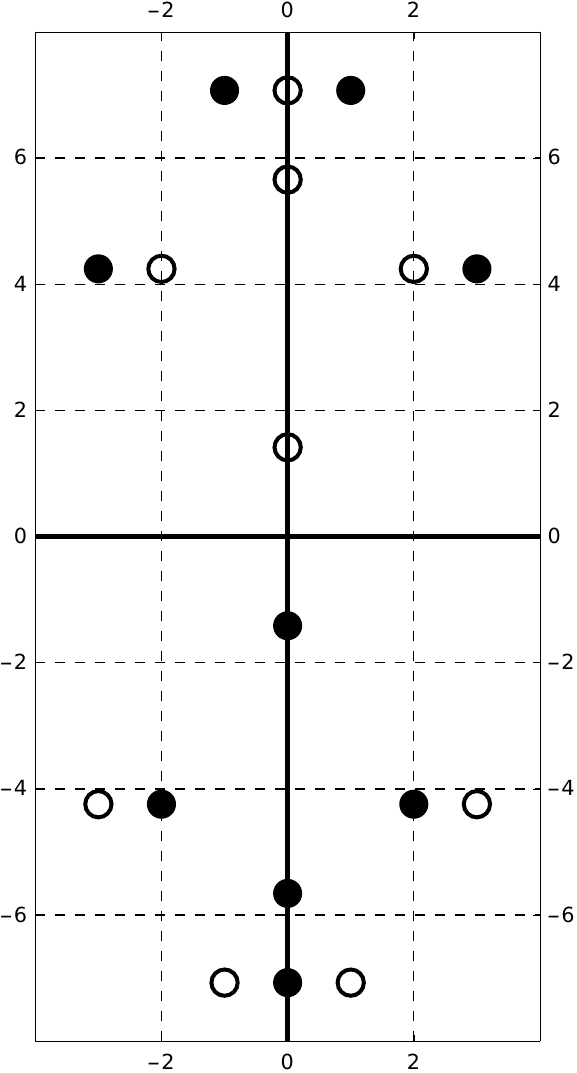}
}
\caption{The symmetric ideal \PTE\ solution \eqref{eqnm29a} in
$\mathbb{Z}[i\sqrt{2}]$ with size $n=9$.}\label{figIS2pics}
\end{figure}

\begin{figure}[tbp]
\centering{
\scalebox{.80}{\begin{tabular}{cc}
\multicolumn{2}{c}{\includegraphics[width=1.75in]{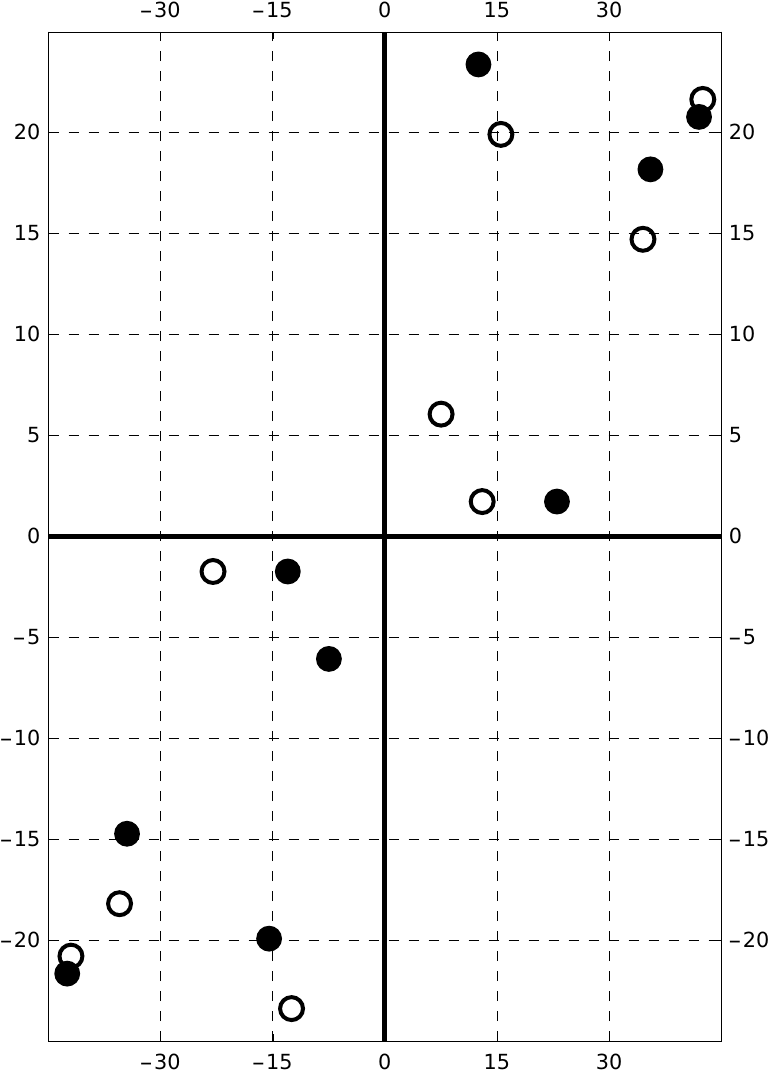}}\\
\multicolumn{2}{c}{(a) $n=9$ \eqref{eqnEI9a}}\\
\includegraphics[width=1.75in]{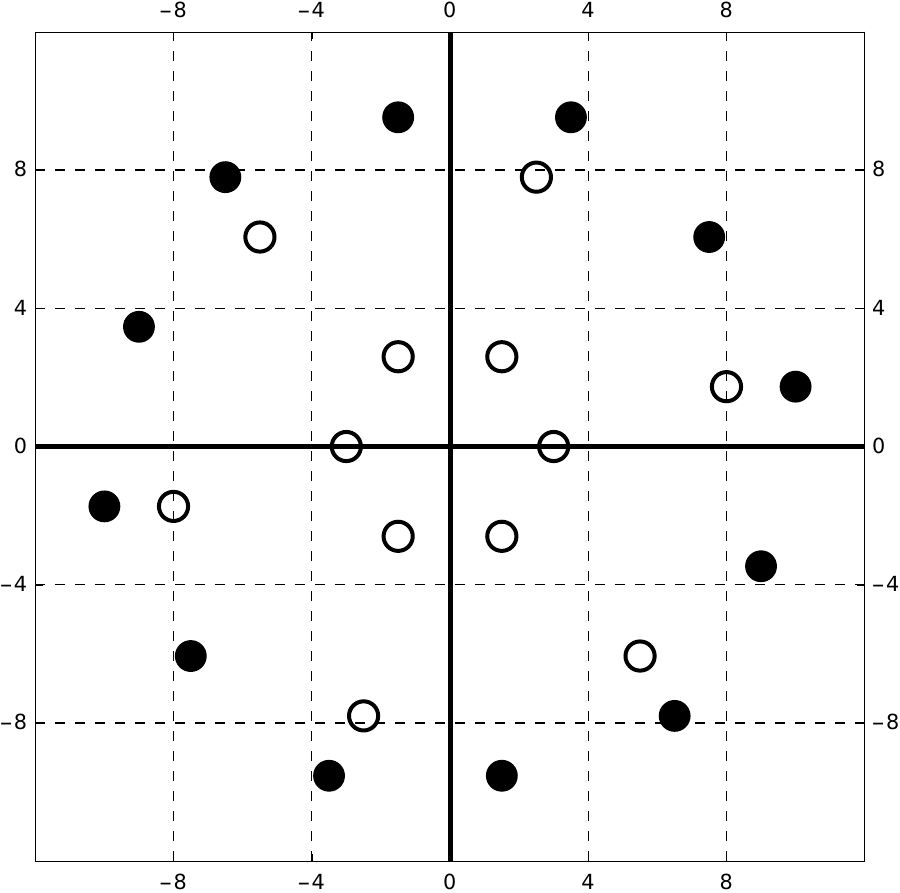} &
\includegraphics[width=1.75in]{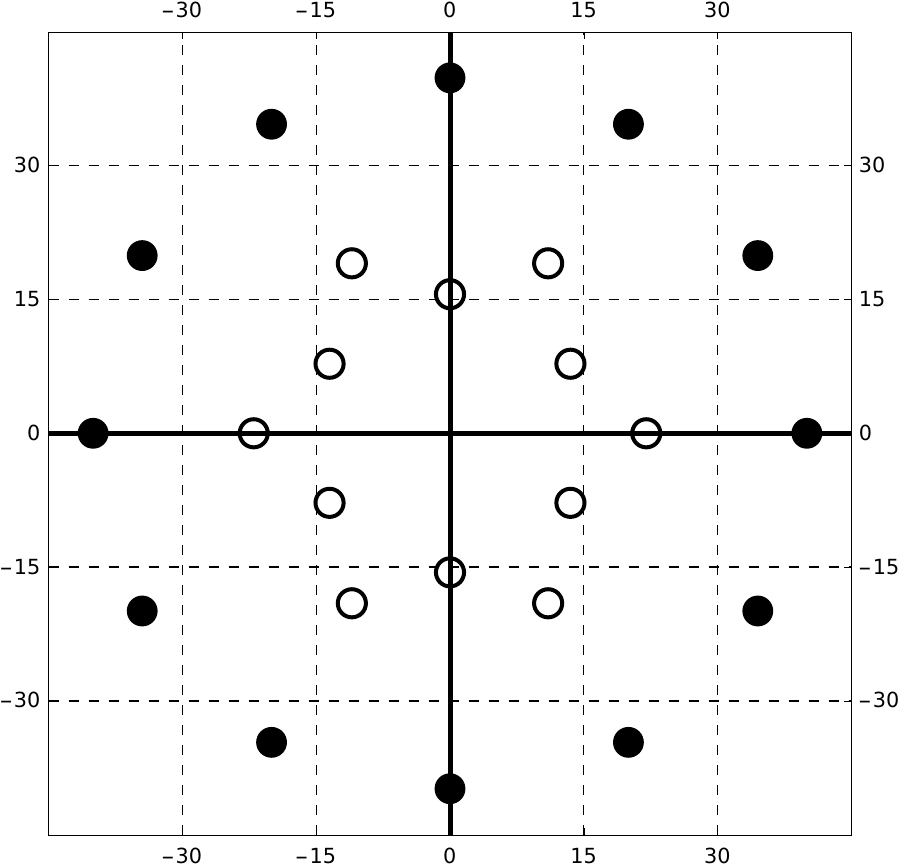}\\
(b) $n=12$ \eqref{eqnEI12a} & (c) $n=12$ \eqref{eqnEI12b}
\end{tabular}}
}
\caption{Symmetric ideal \PTE\ solutions in the Eisenstein
integers with size $n=9$ or $n=12$.}\label{figEIpics}
\end{figure}

\bibliographystyle{plain}

\providecommand{\noopsort}[1]{}

\end{document}